\documentclass[10pt,reqno]{amsart}
\usepackage{amssymb,mathrsfs,color}

\usepackage{wrapfig}
\usepackage{tikz}
\usetikzlibrary{arrows,calc,decorations.pathreplacing}
\definecolor{light-gray1}{gray}{0.90}
\definecolor{light-gray2}{gray}{0.80}

\definecolor{deepgreen}{cmyk}{1,0,1,0.5}

\newcommand{\B}{\mathcal{B}}
\newcommand{\E}{\mathcal{E}}

\newcommand{\HH}{\mathcal{H}}

\newcommand{\cS}{\mathcal{S}}

\newcommand{\N}{\mathbb{N}}
\newcommand{\R}{\mathbb{R}}
\newcommand{\Sp}{\mathbb{S}}
\newcommand{\Z}{\mathbb{Z}}

\newcommand{\al}{\alpha}
\newcommand{\be}{\beta}

\newcommand{\de}{\delta}
\newcommand{\e}{\varepsilon}

\newcommand{\om}{\omega}
\newcommand{\la}{\lambda}
\newcommand{\te}{\theta}

\newcommand{\na}{\nabla}

\newcommand{\supp}{\operatorname{supp}}

\makeatletter

\newcommand{\Rmnum}[1]{\expandafter\@slowromancap\romannumeral #1@}
\makeatother

\newcommand{\I}{\infty}

\newcommand{\ba}{\overline}

\newcommand{\ang}[1]{\left\langle{#1}\right\rangle}
\newcommand{\abs}[1]{\left\lvert{#1}\right\rvert}
\newcommand{\ds}{\displaystyle}

\newcommand{\ant}[1]{\begin{align*}\begin{split} #1 \end{split}\end{align*}}
\newcommand{\EQ}[1]{\begin{equation}\begin{split} #1 \end{split}\end{equation}}

\newcommand{\Del}[1]{}

\numberwithin{equation}{section}

\newtheorem{thm}{Theorem}[section]
\newtheorem{cor}[thm]{Corollary}
\newtheorem{lem}[thm]{Lemma}
\newtheorem{prop}[thm]{Proposition}

\newtheorem{claim}[thm]{Claim}
\theoremstyle{remark}
\newtheorem{rem}{Remark}
\newtheorem{defn}{Definition}

\newcommand{\mif}{{\ \ \text{if} \ \ }}
\newcommand{\mfor}{{\ \ \text{for} \ \ }}
\newcommand{\mas}{{\ \ \text{as} \ \ }}

\renewcommand\Re{\mathrm{Re}\,}

\def\glei{\mathrm{eq}}

\begin{document}

\title[Scattering for the cubic wave equation]{Scattering for the radial $3d$ cubic wave equation}
\author{Benjamin Dodson and Andrew Lawrie}
\begin{abstract} Consider the Cauchy problem for the radial cubic wave equation in $1+3$ dimensions with either the focusing or defocusing sign. This problem is critical in $\dot{H}^{\frac{1}{2}} \times \dot{H}^{-\frac{1}{2}}(\R^3)$ and subcritical with respect to the conserved energy. Here we prove that if the critical norm of a solution remains bounded on the maximal time-interval of existence, then the solution must in fact be global-in-time and scatter to free waves as $t \to \pm \infty$. %A consequence of this result is that there cannot be type-II blow-up. %We note that basic outline of the proof also applies to yield the analogous result for all subcritical powers $p$ with $1 + \sqrt{2}< p  \le 3$.  

\end{abstract}

\thanks{Support of the National Science Foundation, DMS-1103914  for the first author, and DMS-1302782 for the second author, is gratefully acknowledged.}

%The first author gratefully acknowledges support from the National Science Foundation, DMS-1103914. The second author gratefully acknowledges support from the National Science Foundation, DMS-1302782. }

\maketitle

\section{Introduction} 
Consider the Cauchy problem for the cubic  semi-linear wave equation in $\R^{1+3}$, namely, 
\EQ{\label{u gen eq}
&u_{tt}- \Delta u  +   \mu u^3  = 0,\\
&\vec u(0) = (u_0, u_1), 
}
restricted to the radial setting and with $\mu \in \{  \pm 1\}$.  The case $\mu = 1$ yields what is referred to as the defocusing problem since here the conserved energy, 
\EQ{ \label{def foc}
E( \vec u)(t) := \int_{\R^3}  \left[\frac{1}{2}( \abs{u_t(t)}^2 + \abs{\nabla u(t)}^2)  + \frac{1}{4} \abs{u(t)}^4\right] \, dx = \textrm{constant},  
}
is positive for sufficiently regular  non-zero solutions, and the $\dot{H}^1 \times L^2( \R^3)$ norm of a solution, $$\vec u(t) := (u(t), u_t(t)),$$ is bounded by its energy. 

The case $\mu = -1$ gives the focusing problem and the conserved energy for sufficiently regular solutions to~\eqref{u gen eq} is given by 
\EQ{ \label{foc}
E( \vec u)(t) := \int_{\R^3}  \left[\frac{1}{2}( \abs{u_t(t)}^2 + \abs{\nabla u(t)}^2)  - \frac{1}{4} \abs{u(t)}^4\right] \, dx = \textrm{constant}.
}
As we will only be considering radial solutions to~\eqref{u gen eq}, we will often slightly abuse notation by writing $u(t, x) = u(t, r)$ where here  $(r, \om)$, with $r= \abs{x}$, $x= r  \om$,  $\om \in  \Sp^2$, are polar coordinates on $\R^3$. In this setting we can rewrite the Cauchy problem~\eqref{u gen eq} as
\EQ{ \label{u eq}
&u_{tt}- u_{rr}- \frac{2}{r} u_r  \pm u^3 = 0,\\
&\vec u(0) = (u_0, u_1),
}
and the conserved energy (up to a constant multiple) by
 \EQ{\label{E}
E( \vec u)(t) = \int_0^{\infty} \left[ \frac{1}{2}( u_t^2(t) + u_r^2(t))  \pm \frac{1}{4} u^4(t) \right] \, r^2 \, dr.
}
The Cauchy problem~\eqref{u eq} is invariant under the scaling
\EQ{
\vec u(t, r) \mapsto   \vec u_{\la}(t, r) := (\la^{-1} u( t/ \la, r/ \la), \la^{-2} u_t( t/ \la, r/ \la)).
}
One can also check that this scaling leaves unchanged the $\dot{H}^{\frac{1}{2}} \times \dot{H}^{-\frac{1}{2}}$-norm of the solution. It is for this reason that~\eqref{u eq} is called {\em energy-subcritical}.  It is natural to consider the Cauchy problem with initial data $(u_0, u_1) \in\dot{H}^{\frac{1}{2}} \times \dot{H}^{-\frac{1}{2}}$. We remark that~\eqref{u eq} is also invariant under conformal inversion, 
\EQ{
u(t, r) \mapsto \frac{1}{t^2-r^2} u \left(\frac{t}{t^2-r^2}, \, \frac{r}{t^2-r^2} \right).
} 

A standard argument based on Strichartz estimates shows that both the de-focusing and focusing problems are locally well-posed in $\dot{H}^{\frac{1}{2}} \times \dot{H}^{-\frac{1}{2}}(\R^3)$. This  means that for all initial data $\vec u(0) =(u_0, u_1) \in \dot{H}^{\frac{1}{2}} \times \dot{H}^{-\frac{1}{2}}$, there is a unique solution $\vec u(t)$ defined on a maximal interval of existence $I_{\max}$ with $\vec u(t)  \in C(I_{\max} ; \dot{H}^{\frac{1}{2}} \times \dot{H}^{-\frac{1}{2}})$. Moreover, for every compact time interval $J \subset I_{\max}$ we have $u \in S(J):= L^{4}_t(J; L^4_x)$. The Strichartz norm $S(J)$ determines a criteria for both scattering and finite time blow up and we  make these statements precise in Proposition~\ref{small data}. Here we note that in particular one can show that if the initial data $\vec u(0)$ has sufficiently  small  $ \dot{H}^{\frac{1}{2}} \times \dot{H}^{-\frac{1}{2}}$-norm, then the corresponding solution $\vec u(t)$ has finite $S(\R)$-norm and hence scatters to free waves as $t \to \pm \infty$. 

The theory for solutions to~\eqref{u eq} with initial data that is small in $\dot{H}^{\frac{1}{2}} \times \dot{H}^{-\frac{1}{2}}$ is thus very well understood -- all solutions are global-in-time and scatter to free waves as $t \to \pm \infty$. However, much less is known regarding the  \emph{asymptotic dynamics} of solutions to either the defocusing or focusing problems once one leaves the perturbative regime.

%Since this is the scaling critical space for the problem, this means that initial data which are small in $\dot{H}^{\frac{1}{2}} \times \dot{H}^{-\frac{1}{2}}$ lead to global-in-time solutions which scatter to free waves as $t \to \pm \infty$, (see Proposition~\ref{small data} below for a precise statement of these facts). 

It is well known that there are solutions to the focusing problem that blow-up in finite time. To give an example,  
\EQ{ \label{phi}
\phi_T(t, r) = \frac{ \sqrt{2}}{T-t}
}
solves the ODE,  $\phi_{tt} = \phi^3$. Using finite speed of propagation, one can construct from $\phi_T$ a compactly supported (in space) self-similar blow-up solution to~\eqref{u eq}. Indeed define  $\vec{u}_T(t)$, to be the solution to~\eqref{u eq} with initial data $\vec u_T(0, x) = \chi_{2T}(x)  \sqrt{2}$, where $\chi_{2T} \in C_0^{\infty}(\R^3)$, satisfies $\chi_{2T}(x) =1$ if $\abs{x} \le 2T$. Then $\vec u_T(t) = \phi_T(t)$ for all $r \le T$ and $0 \le t < T$ and    blows up at time $t=T$. However, such a self-similar solution  must have its  critical $\dot{H}^{\frac{1}{2}} \times \dot{H}^{-\frac{1}{2}}$--norm tending to $\infty$ as $t \to T$, i.e., 
\ant{
\lim_{t \to T} \| \vec{u}_T(t)\|_{\dot{H}^{\frac{1}{2}} \times \dot{H}^{-\frac{1}{2}}} =  \infty.
}
Indeed, one can show by a direct computation that the $L^{3}(\R^3)$ norm of $\vec u_T(t, x)$ tends to $\infty$ as $t \to T_+$. Since $\dot{H}^{\frac{1}{2}}  \subset L^3$, this means that the $\dot{H}^{\frac{1}{2}} \times \dot{H}^{-\frac{1}{2}}$ norm must blow up as well.  Such behavior is typically referred to as type-I, or ODE blow-up. 

One the other hand, type-II solutions, $\vec u(t)$, are those whose critical norm remains bounded on their maximal interval of existence, $I_{\max}$, i.e., 
\EQ{ \label{type 2}
\sup_{t  \in I_{\max}} \| \vec{u}(t)\|_{\dot{H}^{\frac{1}{2}} \times \dot{H}^{-\frac{1}{2}}} <  \infty.
}
In this paper we restrict our attention to type-II solutions, i.e., those which satisfy~\eqref{type 2}. We prove that if a solution  $\vec u(t)$ to~\eqref{u eq} satisfies~\eqref{type 2}, then $\vec u(t)$  must in fact exist globally-in-time and scatter to free waves in both time directions. %As a consequence, we also obtain the weaker statement that there can be no type-II blow-up for either the focusing or defocusing equation. 
To be precise, we establish the following result. 

\begin{thm}\label{main}
 Let $\vec u(t) \in \dot{H}^{\frac{1}{2}} \times \dot{H}^{-\frac{1}{2}}(\R^3)$ be a radial solution to~\eqref{u eq} defined on its maximal interval of existence $I_{\max}=(T_-, T_+)$. Suppose in addition that  
\begin{equation}\label{1.3}
\sup_{t \in I_{\max}} \| \vec u(t) \|_{\dot{H}^{1/2} \times \dot{H}^{-1/2}(\R^{3})} < \infty.
\end{equation}
Then,  $I_{\max} = \R$, i.e., $\vec u(t)$ is defined globally in time. Moreover,  
\begin{equation}\label{1.4}
\| u \|_{L_{t,x}^{4}( \R^{1+3})} <  \infty, %A\left(\sup_{t \in \R_+} \| \vec u(t) \|_{\dot{H}^{1/2} \times \dot{H}^{-1/2}(\R^{3})}\right).
\end{equation}
which means that $\vec u(t)$ scatters to a free wave in both time directions, i.e., there exist  radial solutions $\vec u_{L}^{\pm}(t)  \in \dot{H}^{\frac{1}{2}} \times \dot{H}^{- \frac{1}{2}}(\R^3)$ to the free wave equation, $\Box u_L^{\pm} = 0$,  so that 
\EQ{
\| \vec u(t)- \vec u^{\pm}_L(t) \|_{\dot{H}^{\frac{1}{2}} \times \dot{H}^{- \frac{1}{2}}(\R^3)} \longrightarrow 0 \mas t \rightarrow + \infty
}
%\noindent That is, either
%\begin{equation}\label{1.5}
%\sup_{t \in I} \| (u(t), u_{t}(t)) \|_{\dot{H}^{1/2}(\R^{3}) \times \dot{H}^{-1/2}(\R^{3})} = \infty,
%\end{equation}
%\noindent or \eqref{u eq} is globally well - posed and scatters to a free solution in both time directions.
\end{thm}

\begin{rem}
Theorem~\ref{main} is a conditional result. Other than the requirement that the initial data be small in $\dot H^{\frac{1}{2}} \times \dot H^{-\frac{1}{2}}$, there is no known general criterion %for large initial data 
that ensures that~\eqref{1.3} is satisfied by the evolution for either the defocusing or the focusing equation. %and we make no claims in this paper towards establishing any such criteria. 
While the methods in this paper apply equally well to both the focusing and defocusing equations one should expect drastically different behavior from generic initial data in these two cases. 

\end{rem}

\begin{rem} The proof of Theorem~\ref{main} readily generalizes to all subcritical powers $p \le 3$ for which there is a satisfactory small data/local well-posedness theory. In particular, the methods presented here allow one to deduce the exact analog of Theorem~\ref{main} for radial  equations~\eqref{semi lin} for all powers $p$ with $1+ \sqrt{2} < p \le 3$ -- here $1+ \sqrt{2}$ is the F. John exponent, see \cite{John, Sch85}. We have chosen to present the details for only the cubic equation to keep the exposition as simple as possible.  We also remark that only the material in Section $4$ relies on the assumption of radiality.  
\end{rem} 

\subsection{History of the problem} 

The cubic wave equation on $\R^{1+ 3}$ has been extensively studied and we certainly cannot give a complete account of the vast body of literature devoted to this problem. 

For the defocusing equation, the positivity of the conserved energy can be used to extend a local existence result to a global one if one  begins  with initial data that is sufficiently regular. In~\cite{Jorg}, Jorgens showed global existence for the defocusing equation for smooth compactly supported data. There has been a good deal of recent work extending the local existence result of Lindblad and Sogge, \cite{LinS}, in $H^s \times H^{s-1}$ for $s > 1/2$ to an unconditional global well-posedness result and we refer the reader to \cite{KPV00, GP03, BC06, Roy09} and the references therein for details. However, since these works are not carried out in the scaling critical space, the issue of global dynamics, and in particular scattering, is not addressed. 

For the focusing equation, type-II finite time blow-up has recently been ruled out for initial data that lies in $\dot{H}^1 \times L^2$ in the work of Killip, Stovall and Visan, ~\cite{KSV}.  There are several works that open up interesting lines of inquiry related to the question of asymptotic  dynamics. In two remarkable works, Merle and Zaag~\cite{MZ03ajm, MZ05ma} determined that all blow-up solutions must blow-up \emph{at the self-similar rate}.  In the radial case, an infinite family of smooth self-similar solutions is constructed  by Bizo\'n et al. in~\cite{BBMW10}. In~\cite{BZ09}, Bizo\'n and  Zengino\u{g}lu give numerical evidence to support a conjecture that a  two parameter family  of solutions, obtained via time translation and conformal inversion of a self-similar solution, serves as a global attractor for  a large set of initial data.  In fact, Donninger and Sch\"orkhuber \cite{DS12} showed that the blow-up profile~\eqref{phi} is stable under small perturbations in the energy topology.

 Equations of the form 
\EQ{ \label{semi lin}
\Box u  =  \pm \abs{u}^{p-1} u
} 
for different values of $p$ and for different dimensions have also been extensively studied. For $d=3$, the energy critical power, $p=5$,  exhibits quite different phenomena than both the subcritical and supercritical equations. Global existence and scattering for all finite energy data was proved by Struwe,~\cite{Struwe88}, for the radial defocusing equation and by Grillakis,~\cite{Gri90}, in the nonradial, defocusing case. 

For the focusing energy critical equation, type-II blow up can occur,  as explicitly demonstrated by Krieger, Schlag, and Tataru~\cite{KST3}, via an energy concentration scenario resulting in the bubbling  off of the ground state solution, $W$, for the underlying elliptic equation; see also \cite{KS12, DHKS, DK}. 

In~\cite{KM08}, Kenig and Merle initiated a powerful program of attack for semilinear equations~\eqref{semi lin}  with the concentration compactness/rigidity method, giving a characterization of possible dynamics for solutions with energy below the threshold energy of the ground state elliptic solution. The subsequent work of Duyckaerts, Kenig, and Merle \cite{DKM1,DKM3, DKM2, DKM4} resulted in a classification of possible dynamics for large energies. In particular, all type-II radial solutions asymptotically resolve into a sum of rescaled solitons  plus a radiation term at their maximal time of existence. Dynamics at the threshold energy of $W$ have been examined by Duyckaerts and Merle~\cite{DM} and above the threshold by Krieger, Nakanishi, and Schlag in~\cite{KNS13AJM, KNS13DCDS, KNS14CMP}. 

Analogues of Theorem~\ref{main} have been established for radial equations with different powers in  $3$ dimensions. Shen proved the exact analog of Theorem~\ref{main} for subcritical powers $3<p<5$ in~\cite{Shen}, and  Kenig, Merle \cite{KM11a}, and then Duyckaerts, Kenig, Merle~\cite{DKM5} established the analog of Theorem~\ref{main} for all supercritical powers $p>5$. 
 Here we address type-II behavior in the remainder of the subcritical range for the radial equation, $1+ \sqrt{2} < p \le 3$. While we focus on the cubic equation, our proof readily generalizes to other subcritical powers. The extra regularity for critical elements proved in Section~\ref{dd} gives an extension and simplification of the argument in~\cite{Shen} which allows us to  treat  the cubic equation and below.  

Leaving the setting of type-II solutions, Krieger and Schlag, \cite{KS14}, have very recently constructed a family of solutions to the supercritical equation, $p>5$, which are smooth, global in time, have infinite critical norm and are stable under small perturbations.  

\subsection{Outline of the the proof of Theorem~\ref{main}}
The proof of Theorem~\ref{main} follows the concentration compactness/ rigidity method developed in~\cite{KM06, KM08}. The proof follows a contradiction argument -- if Theorem~\ref{main} were not true, the linear and nonlinear profile decompositions of Bahouri-G\'erard allow one to construct a minimal solution to~\eqref{u eq}, called the critical element,  which does not scatter -- here the minimality refers to the size of the norm in~\eqref{1.3}. This construction, which is by now standard in the field and  is outlined in Section~\ref{con comp}, yields a critical element whose trajectory in the space $\dot{H}^{\frac{1}{2}} \times \dot{H}^{-\frac{1}{2}}$ is pre-compact up to modulation. The goal is then to prove that this compactness property is too rigid of a property for a nonzero solution and thus the critical element cannot exist.  

A significant hurdle in the way of ruling out a critical element $\vec u_{c}(t)$ for the cubic equation (or any subcritical equation)  lies with the fact that  $\vec u_{c}(t)$ is constructed  in the space $\dot{H}^{\frac{1}{2}} \times \dot{H}^{-\frac{1}{2}}$  and thus useful global monotone quantities that require more regularity, such as the conserved energy and virial type identities,  are not, a priori,  well defined. In general, solutions to the cubic wave equation are only as regular as their initial data as evidenced by the presence of the free propagator $S(t)$ in the Duhamel representation for the solution 
\EQ{\label{duh}
\vec u_c(t_0)  = S(t_0-t) \vec u_c(t)  + \int_{t}^{t_0} S(t_0-s) (0, \pm u^3) \, ds.
}
The critical element is rescued by the fact that the pre-compactness of its trajectory is at odds with the dispersive properties of the free part, $S(t_0-t) \vec u(t)$, and thus the first term on the right-hand-side above is forced to vanish weakly as $t  \to \sup I_{\max} $ and as $t \to \inf I_{\max}$. The second term on the right-hand-side of~\eqref{duh} thus encodes the regularity of the critical element and a gain can be expected due to the presence of the cubic term. The additional regularity is extracted by way of the ``double Duhamel trick," which refers to the consideration of the pairing of 
\ant{
\ang{ \int_{T_1}^{t_0} S(t_0-s) (0, \pm u^3) \, ds, \int^{T_2}_{t_0} S(t_0-\tau) (0, \pm u^3) \, d\tau }_{\dot{H}^1 \times L^2}
} 
where $T_1<t_0$ and $T_2>t_0$.  This technique was developed by Tao in~\cite{Tao07} and utilized in the Kenig-Merle framework for nonlinear Schr{\"o}dinger problems by Killip and Visan~\cite{KV10CPDE, KV10AJM, KVClay},  and for semilinear wave equations  in~\cite{KV11TAMS, Bul12a, Bul12b}. This method is also closely related to the in/out decomposition used by Killip, Tao, and  Visan in~\cite[Section~$6$]{KTV09}. For more details on how to exploit the different time directions above we refer the reader to Section~\ref{dd} and in particular to the proof of Theorem~\ref{t4.1}.

 Indeed we bound the critical element  in $\dot{H}^1 \times L^2$. We then use the conserved energy to rule out a critical element which fails to be compact by a low frequency concentration as such a solution would have vanishing energy, see Section~\ref{cascade}. One is then left with  a critical element that is global-in-time and evolves at a fixed scale. In Section~\ref{sec rigid} we prove that such a solution cannot exist by way of a virial identity. We note that this virial based rigidity argument works for pre-compact solutions to ~\eqref{semi lin} with powers $p  \le 3$, but fails to produce useful estimates for powers $3 <p<5$. However, in this range one can use the ``channels of energy" method pioneered in~\cite{DKM4, DKM5}, see~\cite{Shen}.  For more on this, see Remark~\ref{r:vir}.

\section{Preliminaries}
\subsection{Harmonic analysis}
In what follows we will denote by $P_k$  the usual Littlewood-Paley projections onto frequencies of size $ \abs{ \xi} \simeq 2^k$ and by $P_{ \le k}$ the projection onto frequencies $ \abs{ \xi} \lesssim 2^{k}$. These projections satisfy Bernstein's inequalities.  
\begin{lem}[Bernstein's inequalities]\cite[Appendix A]{Taobook} \label{lem bern} Let $1 \le p \le q \le \infty$ and $s \ge 0$. Let $ f: \R^d \to \R$. Then 
\EQ{ \label{bern}
&\|P_{\ge N} f\|_{L^p} \lesssim N^{-s} \| \abs{\na}^s P_{\ge N} f\|_{L^p},\\
&\|P_{\le N} \abs{\na}^s f\|_{L^p} \lesssim N^{s} \|  P_{\le N} f\|_{L^p}, \, \quad
\|P_{ N} \abs{\na}^{\pm s} f\|_{L^p} \simeq N^{ \pm s} \|  P_{ N} f\|_{L^p}\\
&\|P_{\le N} f\|_{L^q} \lesssim N^{\frac{d}{p}- \frac{d}{q}} \| P_{\le N} f\|_{L^p}, \, \quad
\|P_{ N} f\|_{L^q} \lesssim N^{\frac{d}{p}- \frac{d}{q}} \| P_{N} f\|_{L^p}.
}
\end{lem}
Next, we define the notion of a frequency envelope. 
\begin{defn}\cite[Definition~$1$]{Tao1} \label{freq en}We define a \emph{frequency envelope} to be a sequence  $\be = \{\be_k\}$ of positive real numbers with $\be \in \ell^2$. %and 
%\ant{
%\| \be\|_{\ell^2}  \lesssim B.
%}
Moreover, we require the local constancy condition 
\ant{
2^{-\sigma\abs{j-k}} \be_k \lesssim \be_j \lesssim 2^{\sigma\abs{j-k}} \be_k,
}
where here $\sigma>0$ is a small fixed constant; in what follows we will use $\sigma = \frac{1}{8}$.  If $\be$ is a frequency envelope and $(f, g) \in \dot{H}^s \times \dot{H}^{s-1}$ then we say that \emph{$(f, g)$  lies underneath $\be$} if 
\ant{
\| (P_k f,P_k g)\|_{\dot H^s \times \dot{H}^{s-1}} \le \be_k \, \quad \forall k \in \Z,
} 
and we note that if $(f, g)$ lies underneath $\be$ then we have 
\ant{
\| (f, g)\|_{\dot H^s \times \dot{H}^{s-1}} \lesssim \| \be\|_{\ell^2}.
}
\end{defn}

We will require the following refinement of the Sobolev embedding for radial functions which is a consequence of the Hardy-Littlewood-Sobolev inequality. 
\begin{lem}[Radial Sobolev Embedding]\cite[Corollary A.3]{TVZ}\label{lem rad se} Let  $0 < s <3$ and suppose $f \in \dot{H}^s(\R^3)$ is a radial function. Suppose that 
\ant{
\beta > -\frac{3}{q}, \quad \frac{1}{2} - s \le\frac{1}{q} \le \frac{1}{2}, \quad \frac{1}{q} = \frac{1}{2} - \frac{ \be +s}{3}, 
}
and at most one of the equalities 
$
 q=1, \, \, \, q= \infty, \, \, \,  \frac{1}{q} +s= \frac{1}{2},
 $
 holds. Then 
 \EQ{
 \| r^{\be} f \|_{L^{q}} \le C \|f\|_{\dot{H}^{s}}.
 }
\end{lem}
\subsection{Strichartz estimates}
An essential ingredient for the small data theory are Strichartz estimates for the linear wave equation in $\R^{1+3}$, 
\EQ{ \label{lin wave}
&v_{tt} - \Delta v = F,\\
&\vec v(0) = (v_0, v_1).
}
A free wave will mean a solution to~\eqref{lin wave} with $F=0$ and will be denoted by $\vec u(t) = S(t) \vec u(0)$. In what follows we will say that a pair $(p, q)$ is admissible  if 
\EQ{ \label{admiss} 
p, q \ge 2,  \, \, \frac{1}{p} + \frac{1}{q}   \le \frac{1}{2}, \, \, \, %(p, q) \neq (2, \infty)
}
The Strichartz estimates we state below are standard and we refer the reader to~\cite{Kee-Tao, LinS} or the book~\cite{Sogge} and the references therein for more details. 

\begin{rem}We note that since we will only consider the waves with radial initial data and with $F$ radial, we can allow the endpoint $(p, q) = ( 2, \infty)$ as an admissible pair.  The admissibility of $(2, \infty)$ in the radial setting was established in~\cite{KlaMac93}. This endpoint is of course forbidden for nonradial data in dimension $d=3$. 
\end{rem}
\begin{prop}\cite{Kee-Tao, KlaMac93, LinS, Sogge} \label{prop strich}Let $\vec v(t)$ be a solution to~\eqref{lin wave} with initial data $\vec v(0) \in \dot{H}^s \times \dot{H}^{s-1}(\R^3)$ for $s >0$. Let $(p, q)$, and $(a, b)$ be admissible pairs satisfying the gap condition 
\EQ{
\frac{1}{p} + \frac{3}{q} = \frac{1}{a'} + \frac{3}{b'} - 2  = \frac{3}{2} - s.
}
where $(a', b')$ are the conjugate exponents of $(a, b)$. 
Then, for any time interval $I \ni 0$ we have the estimates 
\EQ{\label{strich}
\| v \|_{L^{p}_t(I; L^q_x)} \lesssim  \| (v_0, v_1) \|_{\dot H^s \times \dot H^{s-1}} + \|F\|_{L^{a'}_t(I; L^{b'}_x)}.
}
\end{prop}

\subsection{Small data theory -- global existence, scattering, perturbative theory}
A standard argument based on Proposition~\ref{prop strich} with $s = 1/2$, $(p, q) = (4, 4)$, and $(a', b') = (4/3, 4/3)$ yields the following small data result. 
\begin{prop}[Small data theory]\label{small data}
Let $\vec u(0) = (u_0, u_1) \in \dot{H}^{\frac{1}{2}} \times \dot{H}^{-\frac{1}{2}}(\R^3)$. Then there is a unique, solution $\vec u(t) \in \dot{H}^{\frac{1}{2}} \times \dot{H}^{-\frac{1}{2}}(\R^3)$ defined on a maximal interval of existence $I_{\max}( \vec u )= (T_-(\vec u), T_+( \vec u))$.  Moreover, for any compact interval $J \subset I_{\max}$ we have 
\ant{ 
\| u\|_{L^4_t(J; L^4_x)} < \infty.
}
Moreover, a globally defined solution $\vec u(t)$ for $t\in [0, \infty)$ scatters as $ t \to \infty$ to a free wave, i.e., a solution $\vec u_L(t)$ of $\Box u_L =0$
if and only if  $ \|u \|_{L^4_t([0, \infty), L^4_x)}< \infty$. In particular, there exists a constant $\de>0$ so that  
\EQ{ \label{global small}
 \| \vec u(0) \|_{\dot{H}^{\frac{1}{2}} \times \dot{H}^{-\frac{1}{2}}} < \de \Longrightarrow  \| u\|_{L^4_t(\R; L^4_x)} \lesssim \|\vec u(0) \|_{\dot{H}^{\frac{1}{2}} \times \dot{H}^{-\frac{1}{2}}} \lesssim \de
 }
and hence $\vec u(t)$  scatters to free waves as $t \to \pm \infty$. Finally, we have the standard finite time blow-up criterion: 
\EQ{ \label{ftbuc}
T_+( \vec u)< \infty \Longrightarrow \|u \|_{L^4_t([0, T_+( \vec u));L^4_x)} = + \infty
}
A similar statement holds if $- \infty< T_-( \vec u)$. 
\end{prop}
For the concentration compactness procedure in Section~\ref{con comp} one requires the following perturbation theory for approximate solutions to~\eqref{u eq}. 
\begin{lem}[Perturbation Lemma] \label{pert}
There are continuous functions \\$\e_0,C_0:(0,\I)\to(0,\I)$ such that the following holds:
Let $I\subset \R$ be an open interval (possibly unbounded), $\vec u, \vec v\in C(I; \HH)$  satisfying for some $A>0$
%\EQ{ \label{asm ebd}
 % \|\vec u\|_{L^\I_t(I;\HH)} + \|\vec v\|_{L^\I_t(I;\HH)}  <\I , \pq \|v\|_{L^3_t(I;L^6_x)} \le B,}
\ant{
%\|\vec u\|_{L^\infty(I;\HH)} + 
 \|\vec v\|_{L^\infty(I;\dot{H}^{\frac{1}{2}} \times \dot H^{-\frac{1}{2}})} +   \|v\|_{L^4_t(I;L^4_x)} & \le A \\
 \|\glei(u)\|_{L^{\frac{4}{3}}_t(I; L^{\frac{4}{3}}_x)}
   + \|\glei(v)\|_{L^{\frac{4}{3}}_t(I; L^{\frac{4}{3}}_x)} + \|w_0\|_{L^4_t(I;L^4_x)} &\le \e \le \e_0(A),
   }
where $\glei(u):=\Box u \pm u^3$ in the sense of distributions, and $\vec w_0(t):=S(t-t_0)(\vec u-\vec v)(t_0)$ with $t_0\in I$ arbitrary but fixed.  Then
\ant{ 
  \|\vec u-\vec v-\vec w_0\|_{L^\infty(I;\dot{H}^{\frac{1}{2}} \times \dot H^{-\frac{1}{2}})}+\|u-v\|_{L^4_t(I;L^4_x)} \le C_0(A)\e.}
  In particular,  $\|u\|_{L^4_t(I;L^4_x)}<\I$.
\end{lem}

\subsection{Blow-up for non-positive energies}
Finally, we recall that in the case of the focusing equation, any nontrivial solution with negative energy must blow-up in both time directions. This result was proved in~\cite{KSV} for solutions to~\eqref{u eq}. 
\begin{prop}\cite[Theorem $3.1$]{KSV}\label{neg blow up} Let $\vec u(t)$ be a solution to~\eqref{u eq} with the focusing sign and with maximal interval of existence $I_{\max} = (T_-, T_+)$. If $E(\vec u) \le 0$ then $\vec u(t)$ is  either identically zero or blows up in finite time in both time directions, i.e., $T_+ < + \infty$ and $T_->- \infty$. 
\end{prop}

\section{Concentration compactness}\label{con comp}
In this section  we begin the proof of Theorem~\ref{main}. We will follow the concentration-compactness/rigidity method introduced by Kenig and Merle in~\cite{KM06, KM08}. The concentration compactness part of the argument, which is based on the profile decompositions of Bahouri and Gerard, \cite{BG}, is by now standard and we will essentially follow the scheme from~\cite{KM10}, which is a refinement of the methods from~\cite{KM06, KM08}. Indeed, the main conclusion of this section is that  in the event that Theorem~\ref{main} fails, there exists a minimal, nontrivial, non-scattering solution to~\eqref{u eq}, which we will call the critical element. 

We begin with some notation, following~\cite{KM10} for convenience.  For initial data $ (u_0, u_1) \in \dot H^{\frac{1}{2}} \times \dot H^{-\frac{1}{2}}$ we let  $\vec u(t) \in \dot H^{\frac{1}{2}} \times \dot H^{-\frac{1}{2}}$ be the unique solution to~\eqref{u eq} with initial data $\vec u(0) = (u_0, u_1)$ defined on its maximal interval of existence \\$I_{\max}(\vec u) := (T_-( \vec u), T_+( \vec u))$.  For $A>0$ define 
\EQ{
\B(A):= \{ (u_0, u_1) \in\dot H^{\frac{1}{2}} \times \dot H^{-\frac{1}{2}}   \quad : \quad      \|\vec u(t)\|_{L^{\infty}_t(I_{\max}(\vec u); \dot H^{\frac{1}{2}} \times \dot H^{-\frac{1}{2}})} \le A\}.
}
\begin{defn} We say that $\mathcal{SC}(A)$ holds if for all $ \vec u = (u_0, u_1) \in \B(A)$ we have $I_{\max}(\vec u) = \R$ and $ \|u \|_{L^4_t L^4_x} < \infty$. We also say that $\mathcal{SC}(A;  \vec u)$ holds if $\vec u \in \B(A)$, $I_{\max} (\vec u) = \R$ and $ \|u \|_{L^4_t L^4_x} < \infty$. 
\end{defn} 
\begin{rem}
Recall from Proposition~\ref{small data} that $\| u \|_{L^4_t; L^4_x}< \infty$ if and only if $\vec u$ scatters to a free waves as $t \to \pm \infty$. Therefore Theorem~\ref{main} is equivalent to the statement that $\mathcal{SC}(A)$ holds for all $A>0$. 
\end{rem} 

Now suppose that Theorem~\ref{main} {\em is false}. By  Proposition~\ref{small data}, there is an $A_0>0$ small enough so that $\mathcal{SC}(A_0)$ holds.  Give that we are assuming that Theorem~\ref{main} fails, we can find a threshold, or critical  value $A_C$ so that for $A<A_C$, $\mathcal{SC}(A)$ holds, and for $A>A_C$, $\mathcal{SC}(A)$ fails. Note that $0<A_0<A_C$. The standard conclusion of this assumed failure of the Theorem~\ref{main}  is that there is a minimal non-scattering solution $\vec u(t)$ to~\eqref{u eq} so that $\mathcal{SC}(A_C, \vec u)$ fails, which enjoys certain compactness properties. 

 We will state a refined version of this result below, and we refer the reader to~\cite{KM10, Shen, TVZ, TVZ08} for the details of the argument. As usual, the main  ingredients are the linear and nonlinear Bahouri-Gerard type profile decompositions from~\cite{BG} used in conjunction with the perturbation theory in Lemma~\ref{pert}. 

\begin{prop} \label{crit element} Suppose that Theorem~\ref{main} is false. Then, there exists a solution $\vec u(t)$ such that $\mathcal{SC}(A_C;  \vec u)$ fails. Moreover, we can assume that $\vec u(t)$ does not scatter in either time direction, i.e.,  
\EQ{\label{blow up}
\|u\|_{L^4((T_-(\vec{u}), 0]; L^4_x)} = \|u\|_{L^4([0, T_+(\vec u)); L^4_x)} =  \infty
}
In addition, there exists a continuous  function $N: I_{\max}(\vec u) \to (0, \infty)$ so that the set 
\EQ{\label{K}
K:= \left\{ \left(\frac{1}{N(t)} u\left( t, \frac{\cdot}{N(t)} \right), \, \frac{1}{N^2(t)} u_t\left( t, \frac{\cdot}{N(t)} \right) \right) \mid t \in I_{\max}\right\}
} 
is pre-compact in $\dot{H}^\frac{1}{2} \times \dot H^{-\frac{1}{2}} (\R^3)$. 
\end{prop}
\begin{rem}
After passing to subsequences,  scaling considerations, and possibly time reversal, we can assume, without loss of generality, that $T_+(\vec u) = + \infty$, and $N(t) \le 1$ on $[0, \infty)$. We can further reduce this to two separate cases: Either we have 
\begin{itemize}
\item $N(t) \equiv 1$ for all $t \ge 0$
\item $\limsup_{t \to \infty} N(t)  = 0$
\end{itemize}
These reductions follow from general arguments and are now standard. See for example \cite{KM10, KTV09, Shen} for more details. 
\end{rem}

In what follows it will be convenient to give a name to the compactness property~\eqref{K} satisfied by the critical element. 
\begin{defn} Let $I \ni 0$ be a time interval and suppose $\vec u(t)$ be a solution to~\eqref{u eq} on  $I$. We will say $\vec u(t)$ has the \emph{compactness property on $I$} if there exists a continuous  function $N: I \to (0, \infty)$ so that the set 
\ant{
K:= \left\{ \left(\frac{1}{N(t)} u\left( t, \frac{\cdot}{N(t)} \right), \, \frac{1}{N^2(t)} u_t\left( t, \frac{\cdot}{N(t)} \right) \right) \mid t \in I_{\max}\right\}
} 
is pre-compact in $\dot{H}^\frac{1}{2} \times \dot H^{-\frac{1}{2}} (\R^3)$. 
\end{defn} 
%If scattering fails then there exists a minimal

%\begin{equation}\label{3.1}
%\| u \|_{L_{t}^{\infty} \dot{H}_{x}^{1/2}(I \times \R^{3})} + \| u_{t} \|_{L_{t}^{\infty} \dot{H}_{x}^{-1/2}(I \times \R^{3})}
%\end{equation}

%\noindent norm solution that blows up. At this size, there exists a function $N(t) : I \rightarrow (0, \infty)$ such that

%\begin{equation}\label{3.2}
%N(t) u(t, x N(t)), N(t)^{2} u_{t}(t, x N(t))
%\end{equation}

%\noindent lies in a compact subset of $\dot{H}^{1/2} \times \dot{H}^{-1/2}$.

\begin{rem}\label{Ceta}
A consequence of a critical element having the \emph{compactness property} on an interval $I$ is that, after modulation,  we can control the $\dot{H}^{\frac{1}{2}} \times \dot{H}^{-\frac{1}{2}}$ tails uniformly in $t \in I$. Indeed, by the Arzela-Ascoli theorem,  for any $\eta > 0$ there exists $C(\eta) < \infty$ such that
\EQ{\label{3.3}
&\int_{|x| \geq \frac{C(\eta)}{N(t)}}||\nabla|^{1/2} u(t,x)|^{2} dx + \int_{|\xi| \geq C(\eta)N(t)} |\xi| |\hat{u}(t,\xi)|^{2} d\xi  \le \eta, \\
&\int_{|x| \geq \frac{C(\eta)}{N(t)}}||\nabla|^{-1/2} u_t(t,x)|^{2} dx + \int_{|\xi| \geq C(\eta)N(t)} |\xi|^{-1} |\hat{u_t}(t,\xi)|^{2} d\xi  \le  \eta, 
}
for all $t \in I$. 
\end{rem}

Another standard fact about solutions to~\eqref{u eq} with the compactness property is that any Strichartz norm of the linear part of the evolution vanishes  as $t \to T_- $ and as $t \to T_+$. A concentration compactness argument then implies that the linear part of the evolution must vanish weakly in $\dot{H}^{\frac{1}{2}} \times \dot{H}^{-\frac{1}{2}}$, i.e, 
for any $t_{0} \in I$,
\begin{equation}
S(t_{0} - t) u(t) \rightharpoonup 0
\end{equation}
weakly in $\dot{H}^{1/2} \times \dot{H}^{-1/2}$ as $t \nearrow \sup I$, $t \searrow \inf I$, 
see  \cite[Section $6$]{TVZ08}\cite[Proposition $3.6$]{Shen}. 
This implies the following lemma, which will be crucial in the  proof of Theorem~\ref{main}. 
\begin{lem} \cite[Section $6$]{TVZ08}\cite[Proposition $3.6$]{Shen} \label{lem weak}Let $\vec u(t)$ be a  solution to~\eqref{u eq} with the compactness property on an interval $I = (T_-, T_+)$. Then for any $t_0 \in I$ we can write 
\EQ{
\int_{t_0}^T S(t_0 - s) (0, \pm u^3)  \, ds \rightharpoonup \vec u(t_0)  \mas T \nearrow T_+ \quad \textrm{weakly in} \, \, \dot{H}^{\frac{1}{2}} \times \dot{H}^{-\frac{1}{2}}\\
-\int^{t_0}_T S(t_0 - s) (0, \pm u^3)  \, ds \rightharpoonup \vec u(t_0)  \mas T \searrow T_- \quad \textrm{weakly in} \, \, \dot{H}^{\frac{1}{2}} \times \dot{H}^{-\frac{1}{2}}
}
\end{lem}

\section{Additional regularity for critical elements}\label{dd}
In this section we show that the critical element $\vec u(t)$ from Section~\ref{con comp} has additional regularity for $t \in I$. In particular, we prove the following result. 

\begin{thm}\label{t4.1} Let $\vec u(t)$ be a solution to~\eqref{u eq} defined on a time interval $I = (T_-, \infty)$ with $T_-<0$  and suppose that $\vec u(t)$ has the compactness property on $I$ with $N(t) \le 1$ for all $t \in [0, \infty)$. Then for each $t \in I$ we have $\vec u(t) \in \dot{H}^1 \times L^2$ and the estimate  
\begin{equation}\label{4.1}
\| \vec u(t) \|_{\dot{H}^{1} \times L^2(\R^{3})}  \lesssim N(t)^{1/2}.
\end{equation}
holds with a constant that is uniform for $t \in I$. 
\end{thm}

\begin{rem} We note that all constants   this section implicit in the symbol $\lesssim$ will be allowed to depend on the $L^{\infty}_t(I;  \dot{H}^{\frac{1}{2}} \times \dot{H}^{-\frac{1}{2}})$ norm of $ \vec u$, which is fixed. 

\end{rem}

We will prove Theorem~\ref{t4.1} using a bootstrap procedure with two steps. In particular, we will first show that if  $\vec u(t)$ has the compactness property on an interval $I$ as in Theorem~\ref{t4.1}, then $\vec u(t)$ must lie in $\dot{H}^{\frac{2}{3}} \times \dot{H}^{-\frac{1}{3}}$. We then use this result to attain Theorem~\ref{t4.1}.  

\begin{prop}\label{h23} Let $\vec u(t)$ be as in Theorem~\ref{t4.1}. Then  for any $t_{0} \in I$, we have 
\begin{equation}\label{4.2}
\| \vec u(t_{0}) \|_{\dot{H}^{\frac{2}{3}} \times \dot{H}^{-\frac{1}{3}}(\R^{3})} \lesssim N(t_{0})^{1/6}.
\end{equation}
\end{prop}
We  momentarily postpone the proof of Proposition~\ref{h23} and use it to deduce Theorem~\ref{t4.1}. 
\subsection{Proof of Theorem~\ref{t4.1} assuming Proposition~\ref{h23}}
The first step is to establish refined  Strichartz estimates. 

%\noindent \emph{Proof:} We begin with the assumption that for any $t \in I$,

%\begin{equation}\label{4.2}
%\| u(t) \|_{\dot{H}^{2/3}(\R^{3})} + \| u_{t}(t) \|_{\dot{H}^{-1/3}(\R^{3})} \lesssim N(t)^{-1/6}
%\end{equation}

\begin{lem}\label{l4.2} Let $\vec u(t)$ satisfy the conclusions of Proposition~\ref{h23}. Then
there exists $\delta > 0$ sufficiently small such that for any $t_{0} \in I$,

\begin{equation}\label{4.3}
\| u \|_{L_{t}^{3} L_{x}^{6}([t_{0} - \delta/ N(t_{0}), t_{0} + \delta /N(t_{0})] \times \R^{3})} \lesssim N(t_{0})^{1/6}.
\end{equation}
\end{lem}

\begin{proof} To simplify notation let $J = [t_{0} - \delta/ N(t_{0}), t_{0} + \delta /N(t_{0})]$. We also let $Y = L_{t}^{\infty} L_{x}^{18/5} \cap L_{t}^{3} L_{x}^{6}$ with the natural norm. Using Strichartz estimates, we have
\EQ{\label{4.4}
&\| u \|_{Y(J \times \R^{3})} \lesssim \| \vec u(t_{0}) \|_{\dot{H}^{2/3} \times \dot{H}^{-1/3}(\R^{3})} %+ \| u_{t}(t_{0}) \|_{\dot{H}^{-1/3}(\R^{3})} 
+ \| u^{3} \|_{L_{t}^{6/5} L_{x}^{3/2}(J \times \R^{3})}\\
 &\lesssim N(t_{0})^{1/6} + (\delta /N(t_{0}))^{1/3} \| u \|_{L_{t}^{3} L_{x}^{6}(J \times \R^{3})}^{3/2} \| u \|_{L_{t}^{\infty} L_{x}^{18/5}(J \times \R^{3})}^{3/2} \\& \lesssim N(t_{0})^{1/6} + \de^{\frac{1}{3}}  \bigg( N(t_{0})^{-1/6} \| u \|_{Y(J \times \R^{3})} \bigg)^2  \| u \|_{Y(J \times \R^{3})}.
}
Choosing  $\delta   = 
\delta(N(t_{0})^{-1/6} \| u(t_{0}) \|_{\dot{H}^{2/3} \times \dot{H}^{-1/3}}) > 0,$  small enough  the lemma follows by a standard bootstrapping argument. We remark that here it is important that the constant in~\eqref{4.2} is uniform in $t_0 \in I$. 
\end{proof}
An immediate consequence of Lemma~\ref{l4.2} is the following estimate. 
\begin{cor} \label{Cor: strich}There exists $\de >0$ so that for each $t_0 \in I$ we have 
\begin{equation}\label{4.6}
\| u^{3} \|_{L_{t}^{1} L_{x}^{2}([t_{0} - \delta /N(t_{0}), t_{0} + \delta/ N(t_{0})] \times \R^{3})} \lesssim N(t_{0})^{1/2}.
\end{equation}
\end{cor}
We are now ready to begin the proof of Theorem~\ref{t4.1} assuming Proposition~\ref{h23}. 
\begin{proof}[Proof that Proposition~\ref{h23} implies Theorem~\ref{t4.1}]
Fix $t_{0} \in I$. By translating in time, we can, without loss of generality assume that $t_{0} = 0$. Let
\begin{equation}\label{4.7}
v = u + \frac{i}{\sqrt{-\Delta}} u_{t}.
\end{equation}
Then we have 
\EQ{\label{v u H1}
\| v(t) \|_{\dot{H}^1} \simeq \| \vec u(t) \|_{\dot{H}^1 \times L^2}
}
And if $\vec u(t)$ solves~\eqref{u eq}, then $v(t)$ is a solution to 
\begin{equation}\label{4.8}
v_{t} = -i \sqrt{-\Delta} v  \pm \frac{ i}{\sqrt{-\Delta}} u^{3}.
\end{equation}
where $+$ above corresponds to the focusing equation and $-$ to the defocusing equation. 
By Duhamel's principle, for any $T$, $T_- < T < 0$,
\begin{equation}\label{4.9}
v(0) = e^{iT \sqrt{-\Delta}} v(T) \pm \frac{i}{\sqrt{-\Delta}} \int_{T}^{0} e^{i \tau \sqrt{-\Delta}} u^{3}(\tau) d\tau.
\end{equation}

Next, we define an approximate identity $\{\psi_M\}_{M>0}$. Indeed, let $\psi  \in C^{\infty}_0(\R^3)$ be a radial function, normalized in $L^1(\R^3)$ so that $\|\psi\|_{L^1} = 1$. Set $\psi_M(x) := M^3 \psi(Mx)$. We then define the operator $Q_M$ given by convolution with $\psi_M$, i.e., 
\EQ{ \label{QM def}
Q_M f (x)  := \int_{\R^3} \psi_M(x-y) f(y) \, dy
}
Of course  $Q_M$ is also a Fourier multiplier operator, given by multiplication on the Fourier side by $ \widehat{\psi_M}$, where $\widehat{ \psi_M }(\xi)= \hat \psi \left( \frac{\xi}{M}\right)$. Since $\psi \in C^{\infty}_0$ we have that $\hat \psi \in \cS(\R^3)$. 

With this set-up, it is clear that it suffices to prove that there exists an $M_0>0$ so that 
\EQ{
\| Q_M v(0)\|_{\dot{H}^1} \lesssim N(0)^{\frac{1}{2}}
}
for all $M \ge M_0>0$ with a constant that is independent of $M$. 
% The ordinary Littlewood - Paley projection operator has a rapidly decreasing kernel $K_{M}(x - y) \lesssim_{L} \frac{M^{3}}{\langle M |x| \rangle^{L}}$. We will take the Fourier multiplier $Q_M$ given by the convolution with $M^{3} \psi(M x)$, where $\| \psi \|_{L^{1}} = 1$, $\psi \in C_{0}^{\infty}(\mathbf{R}^{3})$. Then $Q_M$ is a Fourier multiplier satisfying $P_{M}(\xi) \lesssim_{L} \langle \frac{\xi}{M} \rangle^{L}$.
 
To begin, let $T_-<T_1<0<T_2< \infty$ and let $M$ be a large number to be determined below. By the Duhamel formula we have 
\begin{equation}\label{4.10}
\aligned
\Big\langle Q_M v(0), Q_M v(0) \Big\rangle_{\dot{H}^{1}} = \bigg\langle Q_M \left(e^{i T_{2} \sqrt{-\Delta}} v(T_{2}) \mp \frac{i}{\sqrt{-\Delta}} \int_{0}^{T_{2}} e^{it \sqrt{-\Delta}} u^{3} dt\right), \\ Q_M\left( e^{i T_{1}\sqrt{-\Delta}} v(T_{1}) \pm \frac{i}{\sqrt{-\Delta}} \int_{T_1}^{0} e^{i \tau \sqrt{-\Delta}} u^{3} d\tau\right) \bigg\rangle_{\dot{H}^{1}}.
\endaligned
\end{equation}
where here the bracket $\ang{ \cdot, \cdot}_{\dot{H}^1}$ is the $\dot{H}^1$ inner product, namely $$\ang{f, g}_{\dot{H}^1}  = \Re \int_{\R^3} \sqrt{- \Delta }f  \cdot  \ba{ \sqrt{-\Delta} g}.$$ 
 We start by estimating the term that contains both Duhamel terms: 
\begin{align} \notag
&\abs{\left\langle Q_M\bigg(\frac{i}{\sqrt{-\Delta}} \int_{0}^{T_{2}} e^{it \sqrt{-\Delta}} u^{3}(t) dt\bigg), Q_M \bigg(\frac{i}{\sqrt{-\Delta}} \int_{0}^{T_{1}} e^{i \tau \sqrt{-\Delta}} u^{3}(\tau) d\tau \bigg) \right\rangle_{\dot{H}^{1}}  }\\ &=    \abs{\left\langle Q_M \bigg( \int_{0}^{T_{2}}e^{it \sqrt{-\Delta}} u^{3}(t) \, dt \bigg), \, Q_M \bigg(\int_{T_{1}}^{0}e^{i \tau \sqrt{-\Delta}}  (u^{3}(\tau))\, d\tau  \bigg)\right\rangle_{L^{2}}}. \label{4.11}
\end{align}
With $\de>0$ as in Corollary~\ref{Cor: strich} we use $(\ref{4.6})$ to deduce that 
\begin{equation}\label{4.12}
 \int_{\frac{-\delta }{N(0)}}^{\frac{\delta}{ N(0)}} \| Q_M (u^{3}(t)) \|_{L_{x}^{2}(\R^{3})} dt \lesssim N(0)^{1/2}. 
\end{equation}
Next, define a decreasing, smooth, radial function, $\chi \in C^{\infty}_0(\R^3)$ with $\chi(x)  \equiv 1$ for all $\abs{x} \le 1$ and $\chi(x)  = 0$ if $\abs{x} \ge 2$. 
%\begin{equation}\label{4.13}
%\chi(x) = \begin{cases} 1 & |x| \leq 1, \\ 0 & |x| \geq 2.\end{cases} 
%\end{equation}
Also, let $c > 0$ be a small constant, say $c = \frac{1}{4}$. We have
\begin{multline}\label{4.14}
\left\| Q_M\left(\int_{\delta/ N(0)}^{\infty} e^{it \sqrt{-\Delta}} (1 - \chi(\frac{x}{c |t|})) u^3(t) dt \right) \right\|_{L_{x}^{2}(\R^{3})} \lesssim
\\
\lesssim \int_{\delta /N(0)}^{\infty} \left\| (1 - \chi)(\frac{x}{c |t|}) u^{3}(t) \right\|_{L_{x}^{2}(\R^{3})} dt.
\end{multline}
By the radial Sobolev embedding (i.e.,  Lemma~\ref{lem rad se}) we note that %and Berstein's inequalities (i.e., Lemma~\ref{lem bern} and Lemma~\ref{lem rad se}) we note that 
\begin{equation}\label{4.16}
\| |x|^{3/2} u^{3} \|_{L_{x}^{2}}  \lesssim  \|\abs{x}^{\frac{1}{2}} u\|_{L^6_x}^{3} \lesssim \| u\|_{\dot{H}^{\frac{1}{2}}}
%\lesssim \sum_{N_{1} \leq N_{2} \leq N_{3}} \| |x|^{1/2} (P_{N_{1}} u) \|_{L_{x}^{\infty}} \| |x| (P_{N_{2}} u) \|_{L_{x}^{\infty}} \| P_{N_{3}} u \|_{L_{x}^{2}}
\end{equation}
%\begin{equation}\label{4.17}
%\lesssim \sum_{N_{1} \leq N_{2} \leq N_{3}} \left(\frac{N_{1}}{N_{3}}\right)^{1/2} \| P_{N_{1}} u \|_{\dot{H}^{1/2}} \| P_{N_{2}} u \|_{\dot{H}^{1/2}} \| P_{N_{3}} u \|_{\dot{H}^{1/2}} \lesssim \| u \|_{\dot{H}^{1/2}(\R^{3})}^{3}.
%\end{equation}
Therefore,
\begin{equation}\label{4.18}
\left\| (1 - \chi)(\frac{x}{c|t|}) u^{3} \right\|_{L_{x}^{2}(\R^{3})} \lesssim \frac{1}{ |t|^{3/2}} \| u \|_{\dot{H}^{1/2}}^{3}.
\end{equation}
 Thus,
\begin{equation}\label{4.19}
\int_{\delta /N(0)}^{\infty}\left\| Q_M \Big((1 - \chi)(\frac{x}{c |t|}) u^{3}(t)\Big) \right\|_{L_{x}^{2}(\R^{3})} dt \lesssim  \delta^{-1/2} N(0)^{1/2}.
\end{equation}
The same is also true in the negative time direction. With these estimates in hand,  we write~\eqref{4.11} as a pairing 
\EQ{\label{4.25}
\langle A + B, A' + B' \rangle = \langle A + B, A' \rangle + \langle A, A' + B' \rangle + \langle B, B' \rangle - \langle A, A' \rangle
}
where 
\EQ{\label{A B def 1}
&A := Q_M\bigg(\int_{0}^{\delta/ N(0)} e^{it \sqrt{-\Delta}} u^{3} dt + \int_{\delta /N(0)}^{T_{2}} e^{it \sqrt{-\Delta}} (1 - \chi)(\frac{x}{c|t|}) u^{3} dt\bigg)\\
&B :=  Q_M \bigg(\int_{\frac{\delta}{N(0)}}^{T_{2}} e^{it \sqrt{-\Delta}} \chi(\frac{x}{c|t|}) u^{3}(t)  \, dt\bigg)
}
and $A', B'$ are the corresponding integrals in the negative time direction. We start by estimating the term $\ang{A, A'}$. By~\eqref{4.12} and~\eqref{4.19},
\begin{multline}\label{4.20}
\bigg\langle Q_M\bigg(\int_{0}^{\frac{\delta}{ N(0)}} e^{it \sqrt{-\Delta}} u^{3} dt + \int_{\frac{\delta}{N(0)}}^{T_{2}} e^{it \sqrt{-\Delta}} (1 - \chi)(\frac{x}{c|t|})u^{3} dt\bigg), \\ Q_M\bigg(\int_{\frac{-\delta}{N(0)}}^{0} e^{i \tau \sqrt{-\Delta}} u^{3} d\tau + \int_{T_{1}}^{\frac{-\delta}{N(0)}} e^{i \tau \sqrt{-\Delta}} (1 - \chi)(\frac{x}{c|\tau|}) u^{3} d\tau\bigg) \bigg\rangle_{L^{2}}  \lesssim N(0).
\end{multline}
Next, we examine the term $\ang{B, B'}$, which is given by 
\ant{%\label{4.21}
\int_{T_{1}}^{\frac{-\delta}{N(0)}} \int_{\frac{\delta}{N(0)}}^{T_{2}} \bigg\langle  Q_M\left( e^{it \sqrt{-\Delta}} \chi(\frac{x}{c|t|}) u^{3}(t)\right), Q_M\left(e^{i \tau \sqrt{-\Delta}} \chi(\frac{x}{c |\tau|}) u^{3}(\tau)\right) \bigg\rangle_{L^2} dt d\tau \\ = \int_{T_{1}}^{\frac{-\delta}{N(0)}} \int_{\frac{\delta}{N(0)}}^{T_{2}} \bigg\langle  Q_M\left(\chi(\frac{x}{c|t|}) u^{3}(t)\right), Q_M\left(e^{i (\tau - t) \sqrt{-\Delta}} \chi(\frac{x}{c |\tau|}) u^{3}(\tau)\right) \bigg\rangle_{L^2} dt d\tau.
}
To estimate the above, we begin by noting that  the sharp Huygens principle implies that $$\left(e^{i(t - \tau) \sqrt{-\Delta}} \chi(\frac{\cdot}{c|\tau|}) u^{3}(\tau)\right)(x)$$ is supported on the set $|x| > \frac{3}{4} |t - \tau|$ for $c = \frac{1}{4}$. Also, we note that since $t >0$ and $\tau<0$ we have $|t - \tau| > |t|, |\tau|$. Next, recall that the kernel of $Q_M$ is given by the function $\psi_M(x) = M^{3} \psi(M x)$, where $\psi \in C_{0}^{\infty}$. This implies that for $M \gg N(0)^{-1}$, large enough we have 
\ant{
 &\supp \left(\int_{\R^3}\psi_{M}(x - z)  \chi( \frac{z}{c \abs{t}})u^{3}(t, z)  \, dz\right)  \subset \{\abs{x} < \frac{1}{2} \abs{t}\}\\
 &\supp \left( \int_{\R^3} \psi_M(x-y)\left(e^{i(\tau - t) \sqrt{-\Delta}}  \chi(\frac{\cdot}{c \abs{\tau}})u^{3}(\tau)\right)(y) dy \right) \subset  \{\abs{x} > \frac{1}{2} \abs{ t- \tau} \}
}
%\begin{equation}\label{4.22}
%|K_{M}(x)| \lesssim_{L} \frac{M^{3}}{(1 + M\abs{x})^{L}},
%\end{equation}
Therefore, as long as $M$ is chosen large enough, say for $M \ge M_0 \gg N(0)^{-1}$,  and since $\abs{t}< \abs{t- \tau}$ for $t> \de/N(0)$ and $\tau< - \de/N(0)$, we have 
\begin{multline}\label{4.23}
\bigg\langle \int\psi_{M}(x - z)  \chi( \frac{z}{c \abs{t}})u^{3}(t, z) \, dz,  \\ \int \psi_M(x-y)\left(e^{i(\tau - t) \sqrt{-\Delta}}  \chi(\frac{\cdot}{c \abs{\tau}})u^{3}(\tau)\right)(y) dy \bigg\rangle_{L^2}= 0.
%\\ \lesssim_{L} \| u \|_{L_{t}^{\infty} L_{x}^{3}(I \times \R^{3})}^{6} \frac{1}{|\tau - t|^{L}} M^{3 - L}.
\end{multline}

%\noindent Moreover, for any $L > 2$,

%\begin{equation}\label{4.24}
%\int_{t < -\delta /N(0)} \int_{\tau > \delta /N(0)} \frac{M^{3 - L}}{|t - \tau|^{L}} dt d\tau \lesssim M^{3 - L} N(0)^{L-2}.
%\end{equation}
%\noindent Here take $L = 3$. 
It remains to estimate the terms $\ang{A, A'+B'}$ and $\ang{A+B, A'}$,  which are given by 
\begin{multline}\label{4.26}
 \bigg\langle Q_M\bigg(\int_{0}^{\delta /N(0)} e^{i t \sqrt{-\Delta}} u^{3}(t) dt + \int_{\delta/ N(0)}^{T_{2}} e^{i t \sqrt{-\Delta}} (1 - \chi)(\frac{x}{c |t|}) u^{3}(t) dt \bigg),  \\
 Q_M  \bigg(\int_{T_{1}}^{0} e^{i\tau \sqrt{-\Delta}} u^{3}(\tau) d \tau \bigg)  \bigg\rangle_{L^{2}},
\end{multline}
and
\begin{align}\label{4.27}
&\bigg\langle Q_M\bigg(\int_{0}^{T_{2}} e^{i t \sqrt{-\Delta}} u^{3}(t)\, dt\bigg), \\ &Q_M \bigg(\int_{T_{1}}^{-\delta/ N(0)} e^{i \tau \sqrt{-\Delta}} (1 - \chi)(\frac{\cdot}{c |\tau|}) u^{3}(\tau) d\tau + \int_{-\delta/ N(0)}^{0} e^{i \tau \sqrt{-\Delta}} u^{3}(\tau) \,d\tau\bigg)  \bigg\rangle_{L^{2}}. \notag
\end{align}
We provide the details for how to handle~\eqref{4.26} as the estimates for~\eqref{4.27} are similar. First recall that by the Duhamel principle~\eqref{4.9}, we can write 
\begin{equation}\label{4.28}
Q_M \int_{T_{1}}^{0} e^{i\tau \sqrt{-\Delta}} u^{3}(\tau) d\tau = \mp i \sqrt{-\Delta} Q_M v(0) \pm i \sqrt{-\Delta} e^{i T_{1} \sqrt{-\Delta}} Q_M v(T_{1}).
\end{equation}
Using  again~\eqref{4.12} and~\eqref{4.19} we have ,
\begin{align} \notag
&\abs{\left\langle  \sqrt{-\Delta} Q_M v(0), Q_M \left(\int_{0}^{\frac{\delta}{N(0)}} e^{i t \sqrt{-\Delta}} u^{3} dt + \int_{\frac{\delta}{N(0)}}^{T_{2}} e^{i t \sqrt{-\Delta}} (1 - \chi)(\frac{x}{c |t|}) u^{3} dt\right) \right\rangle_{L^{2}}} \\ & \quad \quad \quad \lesssim N(0)^{1/2} \| Q_Mv(0) \|_{\dot{H}^{1}(\R^{3})}.\label{4.29}
\end{align}
We remark that all of the estimates established so far have been uniform in $T_- < T_{1} < 0 < T_{2} < T_+$. This is important as we will now take limits, $T_1 \searrow T_-$ and $T_2 \nearrow T_+$. Indeed, using the weak convergence result in Lemma~\ref{lem weak} we claim that for any fixed $T_2 \in (0, T_+)$ we have 
\begin{multline}\label{4.30}
 \lim_{T_{1} \searrow T_-} \bigg\langle i \sqrt{-\Delta} e^{i T_{1} \sqrt{-\Delta}} Q_M v(T_{1}), \\ Q_M\bigg(\int_{0}^{\delta/ N(0)} e^{i t \sqrt{-\Delta}} u^{3} dt + \int_{\delta/ N(0)}^{T_{2}} e^{i t \sqrt{-\Delta}} (1 - \chi)(\frac{x}{c |t|}) u^{3} dt\bigg) \bigg\rangle_{L^{2}} = 0.
\end{multline}
In fact,  we note that~\eqref{4.12} and~\eqref{4.19} imply that letting $T_{2} \nearrow T_+$, for $M$ fixed,
\ant{
(-\Delta)^{1/4} Q_M \bigg(\int_{0}^{\delta/ N(0)} e^{i t \sqrt{-\Delta}} u^{3} dt + \int_{\delta/ N(0)}^{T_{2}} e^{i t \sqrt{-\Delta}} (1 - \chi)(\frac{x}{c |t|}) u^{3} dt\bigg)
}
converges in $L^{2}(\R^3)$ to
\ant{
(-\Delta)^{1/4} Q_M \bigg(\int_{0}^{\delta /N(0)} e^{i t \sqrt{-\Delta}} u^{3} dt + \int_{\delta /N(0)}^{T_+} e^{i t \sqrt{-\Delta}} (1 - \chi)(\frac{x}{c |t|}) u^{3} dt\bigg) \in L^{2}(\R^{3}).
}
Therefore, since Lemma~\ref{lem weak} says that $e^{iT_{1} \sqrt{-\Delta}} v(T_{1}) \rightharpoonup 0$ weakly in $\dot{H}^{1/2}(\R^{3})$ as $T_{1} \searrow T_-$, we then have 
\begin{multline}\label{4.33}
\lim_{T_{1} \searrow T_-} \lim_{T_{2} \nearrow T_+} \bigg\langle \sqrt{-\Delta} e^{iT_{1} \sqrt{-\Delta} }Q_Mv(T_{1}), \\ Q_M \bigg(\int_{0}^{\delta /N(0)} e^{i t \sqrt{-\Delta}} u^{3} dt+ \int_{\delta /N(0)}^{T_2} e^{i t \sqrt{-\Delta}} (1 - \chi)(\frac{x}{c |t|}) u^{3} dt\bigg) \bigg\rangle_{L^{2}} = 0.
\end{multline}
Thus we have proved that 
\EQ{ \label{A A' B'}
\abs{\lim_{T_1 \searrow T_-} \lim_{T_2\nearrow T_+} \ang{A, A'+B'}} \lesssim  N(0)^{\frac{1}{2}} \|Q_M v(0)\|_{\dot H^1}.
}
Using an identical argument, we can similarly prove that 
\EQ{\label{A B A'}
\abs{\lim_{T_1 \searrow T_-} \lim_{T_2\nearrow T_+} \ang{A+B, A'}} \lesssim  N(0)^{\frac{1}{2}} \|Q_M v(0)\|_{\dot H^1}.
}
where again here, $A, B, A', B'$ are defined as in~\eqref{A B def 1}. 
%\noindent $(\ref{4.27})$ is similar. Once again use $(\ref{4.28})$. By $(\ref{4.12})$ and $(\ref{4.19})$,
%
%\begin{equation}\label{4.34}
%\aligned
%\langle Q_M(\int_{T_{1}}^{-\delta /N(0)} e^{i t \sqrt{-\Delta}} (1 - \chi)(\frac{x}{c |t|}) u^{3} dt + \int_{-\delta /N(0)}^{0} e^{i t \sqrt{-\Delta}} u^{3} dt), \\ \sqrt{-\Delta} Q_M v(0) \rangle_{L^{2}} \lesssim N(0)^{1/2} \| v(0) \|_{\dot{H}^{1}(\R^{3})}.
%\endaligned
%\end{equation}
%
%\noindent Also since $e^{iT_{2} \sqrt{-\Delta}} v(T_{2}) \rightharpoonup 0$ weakly in $\dot{H}^{1/2}(\R^{3})$ as $T_{2} \nearrow T_+$, for any fixed $T_{1}$, $(\ref{4.12})$ and $(\ref{4.19})$ imply that $(-\Delta)^{1/4} Q_M(\int_{T_{1}}^{-\delta /N(0)} e^{i t \sqrt{-\Delta}} (1 - \chi)(\frac{x}{c |t|}) u^{3} dt + \int_{-\delta/ N(0)}^{0} e^{i t \sqrt{-\Delta}} u^{3} dt)$ lies in $L^{2}(\R^{3})$. Therefore,
%
%\begin{equation}\label{4.35}
%\aligned
%\lim_{T_{2} \nearrow T_+} \langle Q_M(\int_{T_{1}}^{-\delta /N(0)} e^{i t \sqrt{-\Delta}} (1 - \chi)(\frac{x}{c |t|}) u^{3} dt + \int_{-\delta/ N(0)}^{0} e^{i t \sqrt{-\Delta}} u^{3} dt), \\ \sqrt{-\Delta} Q_M e^{iT_{2} \sqrt{-\Delta}} v(T_{2}) \rangle_{L^{2}} =0.
%\endaligned
%\end{equation}
%
Therefore, combining~\eqref{4.20}, \eqref{4.29},~\eqref{A A' B'}, and~\eqref{A B A'}, we have proved that 
\begin{equation}\label{4.36}
\aligned
\abs{\lim_{T_{1} \searrow T_-} \lim_{T_{2} \nearrow T_+} \int_{0}^{T_{2}} \int_{T_{1}}^{0} \langle e^{it \sqrt{-\Delta}} Q_M (u^{3}), e^{i \tau \sqrt{-\Delta}} Q_M (u^{3}) \rangle_{L^{2}} dt d\tau } \\ \lesssim \| v(0) \|_{\dot{H}^{1}(\R^{3})} N(0)^{1/2} + N(0).
\endaligned
\end{equation}

We are left to examine the terms in~\eqref{4.10} (once expanded) that contain at most one Duhamel integral. Here we will rely heavily on the $\dot{H}^\frac{1}{2}$-weak convergence in Lemma~\ref{lem weak}. 

Indeed for a fixed $T_{1}$ and fixed $M$, we see that $\sqrt{-\Delta} Q_M v(T_{1}) \in \dot{H}^{1/2}(\R^{3})$. Therefore by Lemma~\ref{lem weak} we have 
\begin{equation}\label{4.37}
\lim_{T_{2} \nearrow T_+} \left\langle e^{iT_{1} \sqrt{-\Delta}} Q_M v(T_{1}), \, e^{i T_{2} \sqrt{-\Delta}} Q_M v(T_{2}) \right\rangle_{\dot{H}^{1}} = 0.
\end{equation}

Next, for  fixed $T_{1} > T_-$, Corollary~\ref{Cor: strich} and the bound $(\ref{4.6})$ imply that $$\| u^{3} \|_{L_{t}^{1} L_{x}^{2}([T_{1}, 0] \times \R^{3})} < \infty,$$ which in turn implies that $\int_{T_{1}}^{0} Q_M e^{it \sqrt{-\Delta}} u^{3} dt \in \dot{H}^{1/2}(\R^{3})$, here again we are using that the multiplier $\widehat{\psi_M} \in  \cS(\R^3)$. Therefore, Lemma~\ref{lem weak} implies
\begin{equation}\label{4.38}
\lim_{T_{2} \nearrow T_+} \left\langle Q_M \left(\int_{T_{1}}^{0} e^{it \sqrt{-\Delta}} u^{3} dt\right), e^{i T_{2} \sqrt{-\Delta}} Q_{M} v(T_{2}) \right\rangle_{\dot{H}^{1/2}} = 0.
\end{equation}

Finally,  we claim that 
\begin{equation}\label{4.39}
\lim_{T_{1} \searrow T_-} \lim_{T_{2} \nearrow T_+} \left\langle Q_M \left(e^{i T_{1} \sqrt{-\Delta}} v(T_{1})\right), \, Q_M \bigg(\int_{0}^{T_{2}} e^{i \tau \sqrt{-\Delta}} u^{3} d\tau\bigg) \right\rangle_{\dot{H}^{1/2}} = 0
\end{equation}
To see this we use $(\ref{4.28})$. Indeed, using again Lemma~\ref{lem weak} we have
\begin{align} \notag
&\lim_{T_{1} \searrow T_-} \lim_{T_{2} \nearrow T_+} \left\langle Q_M\left(e^{i T_{1} \sqrt{-\Delta}}  v(T_{1})\right), \sqrt{-\Delta} Q_M\left(v(0) - e^{i T_{2} \sqrt{-\Delta}}  v(T_{2})\right) \right\rangle_{\dot{H}^{1/2}}\\
&= \lim_{T_{1} \searrow T_-} \left\langle Q_M \left(e^{i T_{1} \sqrt{-\Delta}} v(T_{1})\right), Q_M \left( \sqrt{-\Delta} v(0) \right)\right \rangle_{\dot{H}^{1/2}} = 0.\label{4.41}
\end{align}
 Therefore, \eqref{4.10} together with~\eqref{4.36}--\eqref{4.41} imply that
\begin{equation}\label{4.42}
 \| Q_Mv(0) \|_{\dot{H}^{1}(\R^{3})}^{2} \lesssim \| Q_M v(0) \|_{\dot{H}^{1}(\R^{3})} N(0)^{1/2} + N(0).
\end{equation}
for all $M\ge M_0$ and with a uniform-in-$M$ constant. We can then conclude that 
\EQ{
 \| Q_Mv(0) \|_{\dot{H}^{1}(\R^{3})} \lesssim N(0)^{\frac{1}{2}}
 } 
uniformly in $M \ge M_0$.  Therefore, $\| v(0) \|_{\dot{H}^{1}(\R^{3})} \lesssim N(0)^{1/2}$. This proves Theorem $\ref{t4.1}$, assuming the conclusions of Proposition~\ref{h23}. 
\end{proof}

\subsection{Proof of Proposition~\ref{h23}}

To complete the proof of Theorem~\ref{t4.1} we prove Proposition~\ref{h23}. We begin with another refined Strichartz-type estimate.

%\begin{thm}\label{t4.3}
%For any $t_{0}$,

%\begin{equation}\label{4.43}
%\| u(t_{0}) \|_{\dot{H}^{2/3}(\R^{3})} \lesssim N(t_{0})^{-1/6}.
%\end{equation}
%\end{thm}

%\noindent \emph{Proof:} We need another Strichartz estimate.

\begin{lem}\label{l4.4}
 Let $\eta>0$. There exists $ \de = \delta(\eta) > 0$ such that for all $t_0 \in I$ we have
\begin{equation}\label{4.45}
\| u \|_{L_{t,x}^{4}([t_{0} - \delta/ N(t_{0}), t_{0} + \delta /N(t_{0})] \times \R^{3})} \lesssim \eta.
\end{equation}
\end{lem}

\begin{proof} Again, without loss of generality suppose that $t_{0} = 0$. Then define the interval $J = [- \de/ N(0), \de /N(0)]$. Using the Duhamel formula  we have
\EQ{\label{4.46}
\| u \|_{L_{t,x}^{4}(J \times \R^{3})}  \le  \|S(t) \vec u(0) \|_{L_{t,x}^{4}(J \times \R^{3})} + \left\|  \int_0^t S(t-s) \, (0, \pm u^3) \, ds \right\|_{L_{t,x}^{4}(J \times \R^{3})} \\
%& \lesssim  \|S(t) u(0) \|_{L_{t,x}^{4}(J \times \R^{3})} + \|
}
We estimate the first term on the right-hand-side of~\eqref{4.46} as follows: First choose $C(\eta)$ as in Remark~\ref{Ceta}, \eqref{3.3}, so that 
\begin{equation}\label{4.47}
\| P_{\geq C(\eta) N(0)} \vec u(0) \|_{\dot{H}^{1/2} \times \dot{H}^{-1/2}(\R^{3})}  \le \eta.
\end{equation}
Note that by compactness $C(\eta)$ above can be chosen uniformly in $t \in I$ which is why it suffices to only consider $t_0=0$ in this argument. Next, we have
\begin{multline}
\|S(t) \vec u(0) \|_{L_{t,x}^{4}(J \times \R^{3})} \lesssim \\ \|S(t) P_{\ge C(\eta) N(0)}\vec u(0) \|_{L_{t,x}^{4}(J \times \R^{3})} + \|S(t) P_{\le C(\eta) N(0)}\vec u(0) \|_{L_{t,x}^{4}(J \times \R^{3})}
\end{multline}
We use~\eqref{4.47} together with Strichartz estimates to handle the first term on the right-hand-side above: 
\EQ{
\|S(t) P_{\ge C(\eta) N(0)}\vec u(0) \|_{L_{t,x}^{4}(J \times \R^{3})} \lesssim \| P_{\geq C(\eta) N(0)} \vec u(0) \|_{\dot{H}^{1/2} \times \dot{H}^{-1/2}(\R^{3})}  \lesssim \eta.
}
 To control the second term we use Bernstein's inequalities,~\eqref{bern} and Sobolev embedding,
\begin{equation}\label{4.48}
\| P_{\leq C(\eta) N(0)} S(t)\vec u(0) \|_{L_{x}^{4}(\R^{3})} \lesssim C(\eta)^{1/4} N(0)^{1/4} \| u(0) \|_{\dot{H}^{1/2}(\R^{3})}.
\end{equation}
Taking the $L^4_t(J)$ norm of both sides above gives 
\EQ{
\|S(t) P_{\le C(\eta) N(0)}\vec u(0) \|_{L_{t,x}^{4}(J \times \R^{3})} \lesssim C(\eta)^{\frac{1}{4}} \de^{\frac{1}{4}} 
}
Next, we use Strichartz estimates on the second term on the right-hand-side of~\eqref{4.46}. 
\EQ{
 \left\|  \int_0^t S(t-s) \, (0, \pm u^3) \, ds \right\|_{L_{t,x}^{4}(J \times \R^{3})} \lesssim \| u^3\|_{L^{\frac{4}{3}}_{t,x}(J \times \R^3)} \lesssim \| u\|^3_{L^{4}_{t,x}(J \times \R^3)}
 }
Combining all of the above we obtain, 
\begin{equation}\label{4.49}
\| u \|_{L_{t,x}^{4}(J \times \R^{3})} \lesssim \eta + C(\eta)^{\frac{1}{4}}\de^{\frac{1}{4}} + \| u\|^3_{L^{4}_{t,x}(J \times \R^3)}
\end{equation}
The proof is concluded using the usual continuity argument after taking $\de$ small enough. 
\end{proof}

We can now prove Proposition~\ref{h23}. 
\begin{proof}[Proof of Proposition~\ref{h23}]  We can again, without loss of generality, just consider the case $t_0 = 0$.  We will prove Proposition~\ref{h23} by finding  a frequency envelope $\al_k(0)$  so that %with  
%$$ \gamma_k(0):=2^{k/6}\al_k(0)$$ so that 
\EQ{ \label{al k}
&\|(P_k u(0), P_k u_t(0))\|_{\dot H^{\frac{2}{3}} \times \dot H^{-\frac{1}{3}}} \lesssim  2^{\frac{k}{6}} \al_k(0)\\
& \|\{ 2^{\frac{k}{6}} \al_k(0)\}_{k \in \Z}\|_{\ell^2} \lesssim N(0)^{\frac{1}{6}}
}
Once we find $\al_k(0)$ satisfying~\eqref{al k}, Proposition~\ref{h23} follows from Definition~\ref{freq en}. With this in mind we first establish the following claim: 
\begin{claim}\label{a al claim} There exists a number $\eta_0>0$ so that the following holds. Let $0<\eta< \eta_0$  and %let 
%Let $\eta>0$ be a small constant and 
let $J: =[-\de/N(0), \de/N(0)],$ where $\de = \de(\eta)>0$ is chosen as in Lemma~\ref{l4.4}. Define
\EQ{\label{4.44}
&a_{k} := 2^{k/2} \| P_{k} u \|_{L_{t}^{\infty} L_{x}^{2}(J)} + 2^{-k/2} \| P_{k} u_{t} \|_{L^{\infty}_t L_{x}^{2}(J)}+ 2^{k/4} \| P_{k} u \|_{L_{t}^{8} L_{x}^{8/3}(J)}, \\
&a_{k}(0) := 2^{k/2} \| P_{k} u(0) \|_{L_{x}^{2}(\R^{3})} + 2^{-k/2} \| P_{k} u_{t}(0) \|_{L_{x}^{2}(\R^{3})}.
}
Next define frequency envelopes $\al_k$ and $\al_k(0)$ by 
\begin{equation}\label{4.54}
\alpha_{k} := \sum_{j} 2^{-\frac{1}{8} |j - k|} a_{j}, \hspace{5mm} \alpha_{k}(0) := \sum_{j} 2^{-\frac{1}{8} |j - k|} a_{j}(0).
\end{equation}
Then, as long as $\eta_0$ is chosen small enough we have  
\EQ{\label{4.53}
a_{k} \lesssim a_{k}(0) + \eta^{2} \sum_{j \geq k - 3} 2^{(k - j)/4} a_{j}.
}
and 
\EQ{\label{4.59}
\alpha_{k} \lesssim \alpha_{k}(0).
}

\end{claim}

\begin{proof}[Proof of Claim~\ref{a al claim}]
 To prove~\eqref{4.53} we note that Strichartz estimates, together with Lemma $\ref{l4.4}$ imply that %and $P_{k}((P_{\leq k - 4} u)^{3}) = 0$, Strichartz estimates imply that
\EQ{\label{4.50}
a_k  & = 2^{k/2} \| P_{k} u \|_{L_{t}^{\infty} L_{x}^{2}(J)} + 2^{k/4} \| P_{k} u \|_{L_{t}^{8} L_{x}^{8/3}(J)} \lesssim\\
& \lesssim 2^{k/2} \| P_{k} u(0) \|_{L_{x}^{2}(\R^{3})}  + 2^{-k/2} \| P_{k} u_{t}(0) \|_{L_{x}^{2}(\R^{3})} +2^{\frac{k}{4}} \| P_{k}(u^{3}) \|_{L_{t}^{8/5} L_{x}^{8/7}(J)} \\
& \lesssim a_{k}(0) + \eta^{2} \sum_{j \geq k - 3} 2^{(k - j)/4} a_{j}.
}
To prove the last line above we note that will suffice, by H\"older's inequality in time and Lemma~\ref{l4.4},  to show that 
\EQ{ \label{lp}
  \| P_{k}(u^{3}) \|_{ L_{x}^{\frac{8}{7}}} \lesssim  \| u\|_{L^{4}}^2 \sum_{j \ge k-3} \| P_{j} u\|_{L^{\frac{8}{3}}}
}
First,  since $P_{k}((P_{\leq k - 4} u)^{3}) = 0$,  we have 
\ant{
\|P_k u^3\|_{L^{\frac{8}{7}}_x} &\lesssim  \| P_k [(P_{\le k-4} u)^2 P_{\ge k-3} u]\|_{L^{\frac{8}{7}}} +\| P_k [(P_{\le k-4} u (P_{\ge k-3} u)^2]\|_{L^{\frac{8}{7}}} \\
& \quad +\| P_k [ P_{\ge k-3} u]^3\|_{L^{\frac{8}{7}}} \\
& \lesssim  \|u\|_{L^4}^2 \| P_{\ge k-3} u\|_{L^{\frac{8}{3}}}
}
where the last inequality follows from the boundedness of $P_k$ on $L^p$ and by Holder's inequality. This proves~\eqref{lp}, and thus we have established ~\eqref{4.53}. 
%We estimate the terms on the right-hand side above as follows, 
%\ant{
%\| P_k [(P_{\le k-4} u)^2 P_{\ge k-3} u]\|_{L^{\frac{8}{7}}} &\lesssim \| [(P_{\le k-4} u)^2 P_{\ge k-3} u]\|_{L^{\frac{8}{7}}}  \\
%& \lesssim \| (P_{\le k-4} u)^2\|_{L^2} \| P_{\ge k-3} u\|_{L^{\frac{8}{3}}} \\
%}
%To prove the last line above we note that since $P_{k}((P_{\leq k - 4} u)^{3}) = 0$ it suffices to show that for each $j \ge k-3$ we have 
%\EQ{
%\|(P_j u)^3\|_{L^{\frac{8}{5}}_tL^{\frac{8}{7}}_x(J)} \lesssim 2^{-j/4}\eta^2 a_j
%}
%and this follows by H\"older's inequality followed by  Lemma~\ref{l4.4}. 

To prove~\eqref{4.59} we use~\eqref{4.53} to obtain
\begin{equation}\label{4.55}
\sum_j 2^{-\frac{1}{8}|j - k|} a_{j} \lesssim \sum_j a_{j}(0) 2^{-\frac{1}{8} |j - k|} + \eta^{2} \sum_{j} 2^{-\frac{1}{8} |j - k|} \sum_{j_{1} \geq j - 3} 2^{(j - j_{1})/4} a_{j_{1}}.
\end{equation}
Reversing the order of summation in the second term above gives
\EQ{ \label{al from a}
&\sum_{j_{1} \leq k} \sum_{j \leq j_{1} + 3} 2^{(j - j_{1})/4} 2^{\frac{1}{8} (j - k)} a_{j_{1}} \lesssim \sum_{j_{1} \leq k} 2^{\frac{1}{8} (j_{1} - k)} a_{j_{1}} \lesssim \alpha_{k},\\
&\sum_{j_{1} > k} \sum_{j \leq j_{1} + 3} 2^{(j - j_{1})/4} 2^{-\frac{1}{8} |j - k|} a_{j_{1}} \lesssim \sum_{j_{1} > k} (2^{-\frac{1}{4} (k - j_{1})} + 2^{-\frac{1}{8} (k - j_{1})}) a_{j_{1}} \lesssim \alpha_{k}.
}
Therefore, $(\ref{4.55})$ implies that
\begin{equation}\label{4.58}
\alpha_{k} \lesssim \alpha_{k}(0) + \eta^{2} \alpha_{k},
\end{equation}
which in turn yields~\eqref{4.59} as long as $\eta>0$ is  small enough. 
%\begin{equation}\label{4.59}
%\alpha_{k} \lesssim \alpha_{k}(0).
%\end{equation}
\end{proof}

We now return to the proof of Proposition~\ref{h23}. We note that the calculation in the proof of Claim~\ref{a al claim} also allows us to deduce that 
\begin{equation}\label{4.60}
2^{k/2}\left\| P_{k} \int_{0}^{\delta/ N(0)} \frac{e^{-it \sqrt{-\Delta}}}{\sqrt{-\Delta}} u^3(t) dt \right\|_{L_{x}^{2}(\R^{3})} \lesssim \eta^{2} \sum_{j \geq k - 3} 2^{(k - j)/4} a_{j}.
\end{equation}

%\noindent The radial Sobolev embedding implies that

%\begin{equation}\label{4.61}
%\| |x| u^{3} \|_{L_{x}^{3/2}(\R^{3})} \lesssim \sum_{N_{1} \leq N_{2} \leq N_{3}} \| (P_{N_{1}} u) \|_{L_{x}^{6}} \| |x| (P_{N_{2}} u) \|_{L_{x}^{\infty}} \| P_{N_{3}} u \|_{L_{x}^{2}} \lesssim \| u \|_{\dot{H}^{1}(\R^{3})}^{3}.
%\end{equation}

%\noindent Since $L_{x}^{3/2} \subset H_{x}^{-1/2}$ this implies

%\begin{equation}\label{4.62}
%\| (1 - \chi)(\frac{x}{c|t|}) u^{3} \|_{\dot{H}_{x}^{-1/2}(\R^{3})} \lesssim \frac{1}{c |t|} \| u \|_{\dot{H}^{1/2}}^{3}.
%\end{equation}

%\noindent Interpolating $(\ref{4.18})$ and $(\ref{4.62})$, for any $0 < s_{0} \leq \frac{1}{2}$,
Next, we claim that for any $s_0 \in (0, \frac{1}{2}]$ we have the  estimate 
\begin{equation}\label{4.63}
\int_{\delta/ N(0)}^{\infty} \left\| \frac{e^{-it \sqrt{-\Delta}}}{\sqrt{-\Delta}} (1 - \chi)(\frac{x}{c|t|}) u^{3}\,  \right\|_{\dot{H}^{\frac{1}{2}+s_0}(\R^{3})} dt \lesssim N(0)^{s_0} \delta^{-s_0}.
\end{equation}
where here $c>0$ is a fixed small constant, ($c= \frac{1}{4}$ will do), and $\chi \in C^{\infty}_0(\R^3)$ is radial, $\chi(x) = 1 $ for all $\abs{x} \le 1$ and $\chi(x) = 0$ for all $\abs{x} \ge 2$. To prove~\eqref{4.63}, we note that by Sobolev embedding
\ant{
\left\| \frac{e^{-it \sqrt{-\Delta}}}{\sqrt{-\Delta}} (1 - \chi)(\frac{x}{c|t|}) u^{3} \right\|_{\dot{H}^{\frac{1}{2}+s_0}} &= \left\| (1- \chi)( \frac{x}{c \abs{t}}) u^3 \right\|_{\dot{H}^{-\frac{1}{2}+s_0}} \lesssim \left\| (1- \chi)( \frac{x}{c \abs{t}}) u^3 \right\|_{L^{p}}
}
where $\frac{1}{p} = \frac{2}{3}- \frac{s_0}{3}$. Then using the radial Sobolev embedding, i.e., Lemma~\ref{lem rad se}, we have 
\ant{
(c\abs{t})^{1+s_0}\left\| (1- \chi)( \frac{x}{c \abs{t}}) u^3 \right\|_{L^{p}} &\lesssim \left\| (1- \chi)^{\frac{1}{3}}( \frac{x}{c \abs{t}}) \abs{x}^{\frac{1+s_0}{3}}u \right\|_{L^{3p}}^3 \lesssim \|u\|_{\dot{H}^{\frac{1}{2}}}^3
}
Hence, 
\ant{
\left\| \frac{e^{-it \sqrt{-\Delta}}}{\sqrt{-\Delta}} (1 - \chi)(\frac{x}{c|t|}) u^{3} \right\|_{\dot{H}^{\frac{1}{2}+s_0}} & \lesssim \abs{t}^{-1-s_0} \| u\|_{L^{\infty}_t\dot{H}^{\frac{1}{2}}}
}
Integrating the above in time from $t= \de/N(0)$ to $t = + \infty$ then yields~\eqref{4.63}. 

Once again by the weak convergence result in Lemma~\ref{lem weak} we have,
\ant{
\langle P_{k} v(0), P_{k} v(0) \rangle_{\dot{H}^{1/2}} = \left\langle P_{k} v(0), P_{k}\left(\lim_{T_{2} \nearrow T_+} \pm \frac{i}{\sqrt{-\Delta}} \int_{0}^{T_{2}} e^{i \tau \sqrt{-\Delta}} u^{3}(\tau) d\tau\right) \right\rangle_{\dot{H}^{1/2}}
}
which for all $T_- < T_{1} < 0$, is equal to 
\begin{align}\label{4.65}
=& \lim_{T_{2} \nearrow T_+} \left\langle  P_{k}\big(e^{iT_{1} \sqrt{-\Delta}} v(T_{1})\big), \frac{\pm i}{\sqrt{-\Delta}} P_{k}\bigg(\int_{0}^{T_{2}} e^{i \tau \sqrt{-\Delta}} u^{3} d\tau\bigg) \right\rangle_{\dot{H}^{1/2}}
\\
&+ \lim_{T_{2} \nearrow T_+} \left\langle \frac{1}{\sqrt{-\Delta}} P_{k}\bigg(\int_{T_{1}}^{0} e^{it \sqrt{-\Delta}} u^{3} dt\bigg),\, \,  \frac{1}{\sqrt{-\Delta}} P_{k}\bigg(\int_{0}^{T_{2}} e^{i \tau \sqrt{-\Delta}} u^{3} d\tau\bigg) \right\rangle_{\dot{H}^{1/2}}.\notag% \label{4.66}
\end{align}
 As $T_{1} \searrow T_-$, we note that $\eqref{4.65} \rightarrow 0$. Indeed, by $(\ref{4.9})$,
\EQ{\label{4.67}
&\lim_{T_{1} \searrow T_-} \lim_{T_{2} \nearrow T_+} \left\langle  P_{k}\big(e^{iT_{1} \sqrt{-\Delta}} v(T_{1})\big), \frac{\pm i}{\sqrt{-\Delta}} P_{k}\bigg(\int_{0}^{T_{2}} e^{i \tau \sqrt{-\Delta}} u^{3} d\tau\bigg) \right\rangle_{\dot{H}^{1/2}} \\
\quad &= \lim_{T_{1} \searrow T_-} \lim_{T_{2} \nearrow T_+} \left\langle  P_{k}(e^{iT_{1} \sqrt{-\Delta}} v(T_{1})), P_{k}\big(v(0) - e^{i T_{2} \sqrt{-\Delta}} v(T_{2})\big) \right\rangle_{\dot{H}^{1/2}} \\
\quad &= \lim_{T_{1} \searrow T_-} \left\langle P_{k}(e^{iT_{1} \sqrt{-\Delta}} v(T_{1})), P_{k} (v(0)) \right\rangle_{\dot{H}^{1/2}} = 0.
}
Therefore,
\begin{multline*}%\label{4.68}
 \quad\left\langle P_{k} v(0), P_{k} v(0) \right\rangle_{\dot{H}^{1/2}} = \\ 
 =\lim_{T_{1} \searrow T_-} \lim_{T_{2} \nearrow T_+}\left\langle \frac{1}{\sqrt{-\Delta}} P_{k}\bigg(\int_{T_{1}}^{0} e^{it \sqrt{-\Delta}} u^{3} dt\bigg), \frac{1}{\sqrt{-\Delta}} P_{k}\bigg(\int_{0}^{T_{2}} e^{i \tau \sqrt{-\Delta}} u^{3} d\tau\bigg) \right\rangle_{\dot{H}^{1/2}}.
\end{multline*}
To estimate the right-hand-side above, we split each term into two pieces and use the identity 
\EQ{\label{ang AB}
\ang{A+B , \, A' +B'} = \ang{A+B , \, A' } + \ang{A, \, A' +B'} - \ang{A , \, A'} + \ang{B , \, B'},
} 
where here 
\ant{
&A := P_{k}\left(\int_{-\delta /N(0)}^{0} \frac{e^{it \sqrt{-\Delta}}}{\sqrt{-\Delta}} u^{3} dt\right) + P_{k}\left(\int_{T_{1}}^{-\delta/ N(0)} \frac{e^{it \sqrt{-\Delta}}}{\sqrt{-\Delta}} (1 - \chi)(\frac{x}{c |t|}) u^{3} dt\right),\\
&B:=P_{k}\left(\int_{T_{1}}^{-\delta/ N(0)} \frac{e^{it \sqrt{-\Delta}}}{\sqrt{-\Delta}} \chi(\frac{x}{c |t|}) u^{3} dt\right),
}
and $A', B'$ are the analogous quantities in the positive time direction. 

We begin by estimating the first two terms on the right-hand-side of~\eqref{ang AB}. In fact, an identical argument applies to both of these terms so we only provide details for the term $\ang{A+B, A'}$. To begin, we note that by~\eqref{4.63} we have 
\EQ{\label{2k6 bk}
 \left\| P_{k}\left(\int_{\frac{\delta}{N(0)}}^{T_+} \frac{e^{i\tau \sqrt{-\Delta}}}{\sqrt{-\Delta}} (1 - \chi)(\frac{x}{c |\tau|}) u^{3} d\tau \right)\right\|_{\dot{H}^{1/2}}  \lesssim 2^{-k/6}b_k
 }
 where $b_k= 2^{-k(s_0(k)-1/6)}N(0)^{s_0(k)}$ and we are free to choose any $s_0 = s_0(k) \in (0, \frac{1}{2}]$. Setting $s_0(k) = \frac{5}{24}$ if $2^k \ge N(0)$ and $s_0(k) = \frac{1}{8}$ if $2^k<N(0)$ we have 
 \EQ{\label{bk def}
 b_k:= \begin{cases} 2^{-k/24}N(0)^{\frac{5}{24}} \de^{-\frac{5}{24}} \mif  2^k \ge N(0)\\ 2^{k/24} N(0)^{\frac{1}{8}} \de^{-\frac{1}{8}} \mif 2^k<N(0) \end{cases}
 }
 Then 
 \EQ{ \label{bk l2}
 \|\{b_k\}\|_{\ell^2} \lesssim N(0)^{\frac{1}{6}}.
 }
Therefore with $b_k$ as in~\eqref{bk def} we can combine~\eqref{4.60} and \eqref{2k6 bk} to deduce that 
\ant{%\label{4.69}
&\left\| P_{k} \left(\int_{0}^{\frac{\delta}{ N(0)}} \frac{e^{i\tau \sqrt{-\Delta}}}{\sqrt{-\Delta}} u^{3} dt\right) \right\|_{\dot{H}^{1/2}} + \left\| P_{k}\left(\int_{\frac{\delta}{N(0)}}^{T_+} \frac{e^{i\tau \sqrt{-\Delta}}}{\sqrt{-\Delta}} (1 - \chi)(\frac{x}{c |\tau|}) u^{3} d\tau \right)\right\|_{\dot{H}^{1/2}} \\ 
& \quad\lesssim \eta^{2} \sum_{j \geq k - 3}2^{(k-j)/4} a_{j} + 2^{-k/6} b_{k}.
}
 Then since
\begin{equation}\label{4.70}
\frac{\pm i}{\sqrt{-\Delta}} \int_{T_{1}}^{0} e^{it \sqrt{-\Delta}} u^{3} dt \rightharpoonup v(0) \quad \textrm{in} \, \,  \dot{H}^{\frac{1}{2}} \mas T_1 \searrow T_-
\end{equation}
we can deduce  the estimate
\begin{align} \notag
&\lim_{T_{1} \searrow T_-} \lim_{T_{2} \nearrow T_+}\ang{A+B ,\,  A'}_{\dot{H}^{\frac{1}{2}}} = \lim_{T_{1} \searrow T_-} \lim_{T_{2} \nearrow T_+} \Bigg\langle P_{k} \bigg(\int_{T_{1}}^{0} \frac{e^{it \sqrt{-\Delta}}}{\sqrt{-\Delta}} u^{3} dt\bigg), \\  \notag
&P_{k}\bigg(\int_{0}^{\delta/ N(0)} \frac{e^{-i\tau \sqrt{-\Delta}}}{\sqrt{-\Delta}} u^{3} d\tau\bigg) + P_k \bigg(\int_{\delta /N(0)}^{T_{2}} \frac{e^{i\tau \sqrt{-\Delta}}}{\sqrt{-\Delta}} (1 - \chi)(\frac{x}{c|\tau|}) u^{3} d\tau \bigg) \Bigg\rangle_{\dot{H}^{1/2}}\\
&\lesssim a_{k}(0) \left(\eta^{2} \sum_{j \geq k - 3} 2^{(k - j)/4} a_{j} + %\delta^{-1/6} N(t_{0})^{1/6} 
2^{-k/6} b_{k}\right).\label{4.72}
\end{align}
An identical calculation in the other time direction gives the same estimate for $\ang{A, A'+B'}$. Next, we estimate $\ang{A, A'}$ again using~\eqref{4.60} and~\eqref{2k6 bk}. We have 
\begin{align} \notag
&\Bigg\langle P_{k}\bigg(\int_{0}^{\delta /N(0)} \frac{e^{i\tau \sqrt{-\Delta}}}{\sqrt{-\Delta}} u^{3} d\tau\bigg) + P_{k}\bigg(\int_{\delta/ N(0)}^{T_{2}} \frac{e^{i\tau \sqrt{-\Delta}}}{\sqrt{-\Delta}} (1 - \chi)(\frac{x}{c |\tau|}) u^{3} d\tau\bigg), \\  \notag
&P_{k}\bigg(\int_{-\delta /N(0)}^{0} \frac{e^{it \sqrt{-\Delta}}}{\sqrt{-\Delta}} u^{3} dt\bigg) + P_{k}\bigg(\int_{T_{1}}^{-\delta/ N(0)} \frac{e^{it \sqrt{-\Delta}}}{\sqrt{-\Delta}} (1 - \chi)(\frac{x}{c |t|}) u^{3} dt\bigg) \Bigg\rangle_{\dot{H}^{1/2}}\\
& \lesssim \left(\eta^{2} \sum_{j \geq k - 3} 2^{(k - j)/4} a_{j}\right)^{2} +  2^{-k/3} b_{k}^{2}.\label{4.74}
\end{align}
 Finally, it remains to estimate $\ang{B, B'}$ which is given by
\begin{align} \notag
&\int_{\frac{\delta}{N(0)}}^{T_{2}} \int_{T_{1}}^{\frac{-\delta}{N(0)}} \left\langle P_{k}\bigg(\frac{e^{it \sqrt{-\Delta}}}{\sqrt{-\Delta}} \chi(\frac{x}{c |t|}) u^{3}(t)\bigg), P_{k}\bigg(\frac{e^{i\tau \sqrt{-\Delta}}}{\sqrt{-\Delta}} \chi(\frac{x}{c |\tau|}) u^{3}(\tau)\bigg) \right\rangle_{\dot{H}^{1/2}} dt d\tau\\
\label{4.76}
&=\int_{\frac{\delta}{ N(0)}}^{T_{2}} \int_{T_{1}}^{\frac{-\delta}{N(0)}}\left\langle  \chi(\frac{ x}{c |t|}) u^{3}(t),\,  P_{k}^2\bigg(\frac{e^{i(\tau - t) \sqrt{-\Delta}}}{\sqrt{-\Delta}} \chi(\frac{ \cdot}{c |\tau|}) u^{3}(\tau) \bigg)(x) \right\rangle_{L^2} dt d\tau
\end{align}
 Here we perform an argument similar to our use of the sharp Huygens principle in the proof of~\eqref{4.23}. The kernel of $P_{k}^{2} e^{i(t - \tau) \sqrt{-\Delta}} (\sqrt{- \Delta})^{-1}$ is given by
\begin{equation}\label{4.76.1}
K_k(x) =K_k(\abs{x})= c\int_{0}^{2\pi} \int e^{i|x| \rho \cos \theta} e^{i( \tau-t) \rho}  \rho^{-1}\phi(\frac{\rho}{2^k}) \rho^{2}  d\rho \sin \te d\theta.
\end{equation}
where the integrand is written in polar coordinates on $\R^3$, where $\rho =  \abs{\xi}$.  The function $\phi( \cdot/2^k)$ above is the Fourier multiplier for the Littlewood-Paley projection,~$P_k$, and its support is contained in   $ \rho \in [2^{k-1}, 2^{k+1}]$.  Integration by parts  $L \in \N$ times in $\rho$ gives the estimates 
\EQ{
\abs{K_k(x-y)} \lesssim_L \frac{2^{2k}}{ \ang{2^k \abs{ (\tau-t) - \abs{x-y}}}^L}.
}
 In~\eqref{4.76} we have $|x| \leq \frac{1}{4} |t|$, $|y| \leq \frac{1}{4} |\tau|$, and therefore $|x - y| \leq \frac{1}{4} | \tau-t|$. Thus we have  $$( \tau- t)  - |x-y| \geq \frac{1}{2} |\tau-t|$$ and hence
  \begin{equation}\label{4.76.2}
\abs{K_k(x-y)} \lesssim_{L} \frac{2^{2k}}{\langle 2^{k} | \tau-t| \rangle^{L}}.
\end{equation}
%Once again the sharp Huygens principle implies that $e^{i(\tau - t) \sqrt{-\Delta}} \chi(\frac{x}{c |\tau|}) u^{3}$ is %supported on $|x| \sim \frac{1}{2} |t - \tau|$, so making the same argument as in $(\ref{4.22})$ - $(\ref{4.24})$, 
 If $2^{k} \gg N(0)$, we use~\eqref{4.76.2} with $L=5$ to obtain
\EQ{\label{4.77}
&\int_{\frac{\delta}{ N(0)}}^{T_{2}} \int_{T_{1}}^{\frac{-\delta}{N(0)}}\left\langle  \chi(\frac{ x}{c |t|}) u^{3}(t),\,  P_{k}^2\bigg(\frac{e^{i(\tau - t) \sqrt{-\Delta}}}{\sqrt{-\Delta}} \chi(\frac{ \cdot}{c |\tau|}) u^{3}(\tau) \bigg)(x) \right\rangle_{L^2} dt d\tau \\
& \lesssim \| u\|_{L^{\infty}_tL^3_x}^6 2^{-3k} N(0)^{3} \lesssim  2^{-3k} N(0)^{3}  \lesssim 2^{-\frac{1}{2}k}N(0)^{\frac{1}{2}}
}
If $2^{k} \lesssim N(0)$,  we use the crude estimate $|K_{k}(x-y)| \lesssim 2^{2k}$ in the $(t, \tau)$ region $|t - \tau| \lesssim 2^{-k}$ and we use~\eqref{4.76.2} with $L=3$ in the region where $\abs{\tau-t} \ge 2^{-k}$. We can then conclude that if $2^{k} \lesssim N(0)$ we have
\EQ{\label{4.78}
\int_{\frac{\delta}{ N(0)}}^{T_{2}} \int_{T_{1}}^{\frac{-\delta}{N(0)}}\left\langle  \chi(\frac{ x}{c |t|}) u^{3}(t),\,  P_{k}^2\bigg(\frac{e^{i(\tau - t) \sqrt{-\Delta}}}{\sqrt{-\Delta}} \chi(\frac{ \cdot}{c |\tau|}) u^{3}(\tau) \bigg)(x) \right\rangle_{L^2} dt d\tau \lesssim 1.
}
Therefore,~\eqref{4.72},~\eqref{4.74},~\eqref{4.77}), and~\eqref{4.78} imply that
\EQ{
a_{k}^2(0) &\lesssim a_k(0) \left( \eta^2 \sum_{j \ge k-3}2^{(k-j)/4} a_j + 2^{-k/6} b_k \right) + \left( \eta^2 \sum_{j \ge k-3}2^{(k-j)/4} a_j \right)^2 \\
& \quad + 2^{-k/3}b_k^2 + \min(2^{-k/2}N(0)^{\frac{1}{2}}, 1)
}
Hence we have 
\EQ{
a_k(0) \lesssim \eta^2 \sum_{j \ge k-3}2^{(k-j)/4} a_j  + 2^{-k/6} b_k + \min(2^{-k/4}N(0)^{\frac{1}{4}}, 1)
}
Using the definitions of $\al_k(0)$, $\al_k$, and~\eqref{al from a} we get 
\ant{
\al_k(0) \lesssim  \eta^2 \al_k + \sum_{j} 2^{-\frac{1}{8} \abs{j-k}}2^{-\frac{1}{6}j} b_j + \sum_{j} 2^{-\frac{1}{8} \abs{j-k}}2^{-\frac{1}{6}j} \min(2^{-\frac{1}{12}j} N(0)^{\frac{1}{4}}, 2^{\frac{1}{6}j})
}
Using~\eqref{4.59} and choosing $\eta$ small enough we then have 
\EQ{
\al_k(0) \lesssim  \sum_{j} 2^{-\frac{1}{8} \abs{j-k}}2^{-\frac{1}{6}j} b_j  +   \sum_{j} 2^{-\frac{1}{8} \abs{j-k}}2^{-\frac{1}{6}j} c_j
}
where the $c_j:= \min(2^{-\frac{1}{12}j} N(0)^{\frac{1}{4}}, 2^{\frac{1}{6}j})$ satisfy  
\EQ{\label{cj def}
 \|\{c_j\}\|_{\ell^2} \lesssim N(0)^{\frac{1}{6}}
}
By Schur's test, using~\eqref{bk l2} and~\eqref{cj def}, we can finally conclude that 
\EQ{
\|2^{\frac{1}{6}k} \al_k(0)\|_{\ell^2} \lesssim N(0)^{\frac{1}{6}}
}
as desired. This finishes the proof since $\al_k(0)$ satisfies~\eqref{al k}. 
%\begin{equation}\label{4.79}
%\alpha_{k}\Red{(0)} \lesssim \beta_{k} 2^{-k/6} + \inf(N(0)^{1/4} 2^{-k/4}, 1) + \eta^{2} \alpha_{k},
%\end{equation}
%where
%\begin{equation}\label{4.80}
%\beta_{k} = \sum_{j} 2^{-\frac{1}{12} |j - k|} b_{j}.
%\end{equation}
% This implies that
%\begin{equation}\label{4.81}
%\| u(0) \|_{\dot{H}^{2/3}(\R^{3})} \lesssim N(0)^{1/6}.
%\end{equation}
% Since any point in time can be moved to the origin, this implies Proposition~\ref{h23}. 
\end{proof}

\section{No energy cascade and even more regularity when $N(t) \equiv1$}  
In this section we begin by showing that an energy cascade, i.e., the case, $\limsup_{t \to \infty} N(t) = 0$, is impossible. This leaves us with the soliton-like critical element, $N(t)  = 1$ for all $t \ge0$. We then can reduce this situation to the case of a soliton-like critical element that is global in both time directions with $N(t) \equiv 1$ for all $t \in \R$. Finally, we show that such a solution is in fact uniformly bounded in $\dot{H}^2 \times \dot{H}^1$, which in turn means that $\vec u(t)$ satisfies the compactness property in $\dot{H}^1 \times L^2$. 
\subsection{No Energy Cascade}\label{cascade}
We can quickly rule out the case of a critical element $\vec u(t)$ with scale $N(t)$ satisfying $\limsup_{t \to \infty} N(t) = 0$. We prove the following consequence of Theorem~\ref{t4.1} and Proposition~\ref{neg blow up}. 
\begin{lem} Let $\vec u(t)$ be a solution to~\eqref{u eq} defined on a time interval $I = (T_-, + \infty)$ with $T_-<0$ and  and suppose that $\vec u(t)$ has the compactness property on $I$ with $N(t) \le 1$ for all $t \in [0, \infty)$. Then $\ds{\limsup_{t \to \infty} N(t)  = 0}$ is impossible unless $\vec u(t)~\equiv~0$. 
\end{lem}
\begin{proof} Since $\vec u(t)$ satisfies the conditions of Theorem~\ref{4.1} we see that 
\EQ{ \label{h1 0}
\limsup_{t \to \infty} \|\vec u(t)\|_{ \dot{H}^1 \times L^2}   \lesssim  \limsup_{t \to \infty} N(t)^{\frac{1}{2}} = 0
}
By Sobolev embedding and interpolation we also have 
\EQ{\label{l4 0}
\limsup_{t \to \infty}\| u(t) \|_{L^4} &\lesssim \limsup_{t \to \infty}\| u(t) \|_{ \dot{H}^{\frac{3}{4}}} \lesssim \limsup_{t \to \infty}\| u(t) \|_{ \dot{H}^{\frac{1}{2}}}^{\frac{1}{2}}\| u(t) \|_{ \dot{H}^1}^{\frac{1}{2}}\\& \lesssim \limsup_{t \to \infty}N(t)^{\frac{1}{4}} =0
}
Therefore the conserved energy $E(\vec u(t))$ is well-defined and~\eqref{h1 0} and~\eqref{l4 0} imply that we must have $E(\vec u(t))  = 0$. If $\vec u(t)$ solves the de-focusing equation then $E(\vec u(t))$ is given by~\eqref{def foc} and we can directly conclude that we must have $\vec u(t) \equiv 0$. If $\vec u(t)$ is a solution to the focusing equation then we use Proposition~\ref{neg blow up} to deduce that $\vec u(t) \equiv 0$. 
\end{proof}

\subsection{Additional regularity for a soliton-like critical element}
For the case of a soliton-like critical element, i.e., $N(t) \equiv 1$,  the rigidity argument in Section~\ref{sec rigid} will require that the trajectory  $\vec u(t)$ is pre-compact in $\dot{H}^{1} \times L^{2}(\R^{3})$ rather that just uniformly bounded in this norm, in time. This is not hard to do given our work in the previous sections. 

Let $\vec u(t)$ be as in Proposition~\ref{crit element} and assume that $N(t)  = 1$ for all $t \in [0, \infty)$. Then, without loss of generality, we can assume that $I_{max}(  \vec u) = \R$ and we have $N(t) \equiv 1$ for all $t \in \R$. Indeed, let $t_n \to \infty$ be any sequence. Since $\vec u(t)$ has the compactness property on $(T_-( \vec u), \infty)$ we can find a subsequence, still denoted by $t_n$ so that $\vec u(t_n)  \to  \vec u_{\infty}$ in $\dot{H}^{\frac{1}{2}} \times \dot{H}^{-\frac{1}{2}}$. Then, using the perturbation theory, one can readily check   that the solution $\vec u_{\infty}(t)$ with initial data $\vec u_{\infty}(0) =  \vec u_{\infty}$ is global-in-time, and has the compactness property on $\R$ with $N(t) =1$ for all $t \in \R$. 

We can now establish the following proposition. 
\begin{prop}\label{h2 bound} Let $\vec u(t)$ be the critical element and assume further that $\vec u(t)$ is soliton-like, i.e., $\vec u(t)$ is defined globally-in-time and $N(t) \equiv 1$. Then the trajectory 
\EQ{
K:= \{ \vec u(t) \mid t \in \R\}
}
is pre-compact in $(\dot{H}^{\frac{1}{2}} \times \dot H^{-\frac{1}{2}}) \cap (\dot H^1 \times L^2)(\R^3)$. 
\end{prop}
\begin{proof} We prove that in fact we have a uniform-in-time bound on the $\dot H^2 \times \dot H^1$ norm of $\vec u(t)$. We only provide a sketch of this fact as the proof is nearly identical to the proof of Theorem~\ref{t4.1}. The pre-compactness of $\{\vec u(t) \mid t \in \R\}$ in $\dot{H}^1 \times L^2$ then follows from its  pre-compactness in $\dot H^{\frac{1}{2}}\times \dot H^{-\frac{1}{2}}$  and interpolation as we have 
\ant{
\| \vec u(t) \|_{\dot{H}^1 \times L^2 } \lesssim   \| \vec u(t)\|_{\dot H^{\frac{1}{2}}\times \dot H^{-\frac{1}{2}}}^{\frac{2}{3}} \| \vec u(t) \|_{ \dot H^2 \times \dot H^1}^{\frac{1}{3}}
}
First note that by Theorem~\ref{t4.1} we have 
\EQ{ \label{h1bound}
\| \vec u(t) \|_{\dot{H}^1 \times L^2 } \lesssim 1
}
\begin{claim} \label{endpoint claim} There exists a $\de>0$ so that for all $t_0 \in R$ and for $J:=(t_0- \de, t_0+ \de)$  we have 
\EQ{ \label{endpoint}
\| u\|_{L^2_tL^{\infty}_x(J \times \R^3)}   \lesssim 1
}
\end{claim}

\begin{rem} In~\eqref{endpoint} we make use of the endpoint $L^{2}_tL^{\infty}_x$ Strichartz estimate, which is valid in the radial setting, see~\cite{KlaMac93}. However, this use of the endpoint is for convenience only, as it will allow for an upgrade of the uniform bound in $\dot{H}^1 \times L^2$ directly to a uniform bound in $\dot{H}^2 \times \dot{H}^1$. This implies that the trajectory is pre-compact in $\dot{H}^1 \times L^2$ using the pre-compactness in $\dot{H}^{\frac{1}{2}} \times \dot{H}^{-\frac{1}{2}}$ and interpolating with the $\dot{H}^2 \times \dot H^1$ bound. As we are only interested in proving the compactness property in $\dot{H}^1 \times L^2$ it would also suffice to prove a uniform bound in $\dot{H}^{1+ \e} \times \dot{H}^{\e}$, and for this estimate we would not need the endpoint Strichartz estimate. 
%Say here how using the endpoint estimate make life easy but it not necessary as we only need to show that $\vec u(t)$ is uniformly bounded in $\dot{H}^{1+ \e} \times \dot{H}^{\e}$ to get compactness in $\dot H^1 \times L^2$
\end{rem}
\begin{proof}[Proof of Claim~\ref{endpoint claim}]
First we note that it suffices to prove the claim for $t_0 =0$. We apply the endpoint Strichartz estimates, which are valid in the radial setting. Indeed, denote by $Z(J)$ the space $Z(J):=L^{\infty}_t (J;  \dot{H}^1 \times L^2) \cap L^2_t(J; L^{\infty}_x)$. Then we have 
\EQ{
\| u\|_{Z(J)} &\lesssim  \| \vec u(0)\|_{\dot{H}^1 \times L^2} + \| u^3\|_{L^1_t(J; L^2_x)}  
 \lesssim \| \vec u(0)\|_{\dot{H}^1 \times L^2} + \| u\|_{L^3_t(J; L^6_x)}^3\\
& \lesssim  \| \vec u(0)\|_{\dot{H}^1 \times L^2} +  \de \| u\|_{L^{\infty}_t(J; \dot{H}^1_x)}^3 \lesssim \| \vec u(0)\|_{\dot{H}^1 \times L^2} +  \de \| u\|_{Z(J)}
}
where we remark that we have used the Sobolev inequality and the length of $J$ in the third inequality above, and nothing else. In the last inequality we have used~\eqref{h1bound}. Choosing $\de = \de( \| \vec u(0)\|_{\dot{H}^1 \times L^2}) >0$ small enough completes the proof. Note that here it is important that the constant in~\eqref{h1bound} is uniform in $t_0 \in I$. 
\end{proof}
The proof of Proposition~\ref{h2 bound} now proceeds exactly as in the proof of Theorem~\ref{t4.1} except here we seek an $\dot{H}^2$ bound. We give a brief sketch. Let $\vec v(t)$ be defined as in~\eqref{4.7}, and $Q_M$ as in~\eqref{QM def}. We prove that 
 \EQ{
 \ang{ Q_M v(t_0), Q_M v(t_0)}_{\dot{H}^2} \lesssim 1
 } 
 for all $M \ge M_0$ with a constant that is uniform in $M$ and in $t_0 \in \R$. Extracting weak limits using Lemma~\ref{lem weak} as in the proof of Theorem~\ref{t4.1}, and we note that it will suffice to prove the following estimate for the ``double Duhamel" term: 
\EQ{
\abs{ \ang{ Q_M \bigg( \int_{T_1}^0  e^{i t  \sqrt{- \Delta}}  \na (u^3)(t) \, dt \bigg) , Q_M \bigg( \int^{T_2}_0  e^{i \tau  \sqrt{- \Delta}}  \na (u^3)(\tau) \, d\tau \bigg)}_{L^2}} \lesssim 1
}
where $T_1<0$ and $T_2>0$ and the constant above is uniform in such $T_1, T_2$. Note also that above we have set $t_0 = 0$ as again this case will be sufficient.   

By~\eqref{endpoint} we see that for $\de>0$ as in Claim~\ref{endpoint claim} we have 
\EQ{ \label{0 to de}
\left\| Q_M \bigg( \int^{\de}_0  e^{i \tau  \sqrt{- \Delta}}  \na (u^3)(\tau) \, d\tau  \bigg)\right\|_{L^2} &\lesssim  \int_0^{\de}  \| \na (u^3)\|_{L^2}\\
& \lesssim \| \na u\|_{L^{\infty}_tL^2} \| u\|_{L^{2}_t([0, \de); L^{\infty}_x)}^2 \lesssim 1
}
%We begin by noting that since $N(t) = 1$, we have 
%$$\| u(t) \|_{\dot{H}^{1} \times L^{2}(\R^{3})} \lesssim 1.$$ 
%Indeed,
%\begin{equation}\label{6.1}
%\| \nabla u^{3} \|_{L_{t}^{1} L_{x}^{2}([-\delta, \delta] \times \R^{3})} \lesssim \| \nabla u \|_{L_{t}^{\infty} L_{x}^{2}([-\delta, \delta] \times \R^{3})} \| \langle \nabla \rangle^{1/2} u \|_{L_{t,x}^{4}([-\delta, \delta] \times \R^{3})}^{2} \lesssim 1.
%\end{equation}
% Make the same argument as in the proof of lemma $\ref{l4.4}$ to prove
%\begin{equation}\label{6.2}
% \| \langle \nabla \rangle^{1/2} u \|_{L_{t,x}^{4}([-\delta, \delta] \times \R^{3})} \lesssim 1.
%\end{equation}
Next, by the radial Sobolev embedding, $\| |x|^{3/4} u \|_{L_{x}^{\infty}(\R^{3})} \lesssim \| u \|_{\dot{H}^{3/4}(\R^{3})}$, we have 
\begin{equation}\label{6.3}
\left\| (1 - \chi)(\frac{x}{c |t|}) \nabla u^{3}(t) \right\|_{L^2} \lesssim \frac{1}{c^{3/2} |t|^{3/2}} \| \nabla u \|_{L^{\infty}_t L^{2}(\R^{3})} \| u(t) \|_{\dot{H}^{3/4}}^{2} \lesssim  \abs{t}^{-\frac{3}{2}},
\end{equation}
where $\chi \in C^{\infty}_0(\R^3)$, radial, satisfies $\chi(x) =1$ for $\abs{x} \le 1$ and $\chi(x) = 0 $ for $\abs{x}  \ge 2$, and $c= 1/4$. Therefore we have 
\begin{multline} \label{de to inf}
\left\| Q_M \bigg( \int^{T_2}_\de  e^{i \tau  \sqrt{- \Delta}} (1- \chi(\frac{\cdot}{c \abs{\tau}})) \na (u^3)(\tau) \, d\tau  \bigg)\right\|_{L^2} \lesssim\\
\lesssim \int^{\infty}_\de \|(1- \chi(\frac{\cdot}{c \abs{\tau}})) \na (u^3)(\tau)\|_{L^2} \lesssim  \de^{-1/2}. 
\end{multline}
 Next, using the sharp Hyugens principle exactly as in the proof of~\eqref{4.23} the term 
 \EQ{ \label{sharp h}
&\ang{ Q_M \bigg(   e^{i t  \sqrt{- \Delta}}   \chi( \frac{\cdot}{ c \abs{t}})\na (u^3)(t) \, dt \bigg) , Q_M \bigg(   e^{i \tau  \sqrt{- \Delta}}   \chi(\frac{\cdot}{c \abs{\tau}})\na (u^3)(\tau) \, d\tau \bigg)}\\ 
 &=\ang{ Q_M \bigg(    \chi( \frac{\cdot}{ c \abs{t}})\na (u^3)(t) \, dt \bigg) , Q_M \bigg(   e^{i (\tau-t)  \sqrt{- \Delta}}   \chi(\frac{\cdot}{c \abs{\tau}})\na (u^3)(\tau) \, d\tau \bigg)} \\
 }
 is identically $=0$  for $t < -\de$ and $\tau> \de$.  With~\eqref{0 to de},~\eqref{de to inf}, and~\eqref{sharp h} playing the roles of~\eqref{4.12},~\eqref{4.19}, and~\eqref{4.23}, the proof now proceeds exactly as the proof of Theorem~\ref{t4.1}. We omit the details. 
 %Finally,
%\begin{equation}\label{6.4}
% \| \nabla u^{3} \|_{L_{x}^{1}(\R^{3})} \lesssim \| \nabla u \|_{L_{x}^{2}(\R^{3})} \| u \|_{L_{x}^{4}(\R^{3})}^{2} \lesssim \| u \|_{\dot{H}^{3/4}(\R^{3})}^{2} \| u \|_{\dot{H}^{1}(\R^{3})} \lesssim 1.
%\end{equation}
% Therefore, making an argument identical to the proof of theorem $\ref{t4.1}$, we have that in this case
%\begin{equation}\label{6.5}
 %\| u(t) \|_{\dot{H}^{2}(\R^{3})} + \| u_{t}(t) \|_{\dot{H}^{1}(\R^{3})}
%\end{equation}
%\noindent is uniformly bounded for all $t$. Since $N(t) \equiv 1$ this implies that $(u(t), u_{t}(t))$ lies in a compact subset of $\dot{H}^{1}(\R^{3}) \times L^{2}(\R^{3})$.%
\end{proof}

\section{Rigidity via a Virial Identity}\label{sec rigid}
In this section we complete the rigidity argument by proving that a soliton-like critical element, (i.e., $N(t) \equiv 1$) cannot exist. Indeed we prove the following proposition: 
\begin{prop} \label{rig prop}
Let $\vec u(t) \in ( \dot{H}^{\frac{1}{2}} \times \dot{H}^{-\frac{1}{2}} )\cap (\dot{H}^1 \times L^2)( \R^3)$ be a global-in-time solution to~\eqref{u eq} such that the trajectory 
\EQ{
K := \{ \vec u(t) \mid t \in \R\}  
}
is pre-compact in $( \dot{H}^{\frac{1}{2}} \times \dot{H}^{-\frac{1}{2}} )\cap (\dot{H}^1 \times L^2)( \R^3)$. Then $u(t) \equiv 0$. 
\end{prop}

The proposition will follow from a simple argument based on the following viral identity. In what follows we will fix a smooth radial cutoff function $\chi \in C^{\infty}_0( \R^3)$ so that $\chi(r)  \equiv 1$ for $0 \le r \le 1$, $\supp \chi \subset [0, 2]$, and $\abs{\chi'(r)} \le C$ for all $r > 0$. For each fixed $R>0$ we will denote by $\chi_R$ the rescaling 
\EQ{
\chi_R(r):= \chi( r/R)
}
\begin{lem}[virial identity] Let $\vec u(t) \in  \dot{H}^1 \times L^2(\R^3)$ be a solution to~\eqref{u eq}. Then for every $R>0$ we have 
 \begin{multline}\label{virial}
  \frac{d}{dt} \ang{ u_t \mid \chi_R(u + ru_r)} =    - E( \vec u)(t) + \int_0^{\infty}(1- \chi_R)\left(\frac{1}{2} u_t^2 + \frac{1}{2}u_r^2  \pm \frac{1}{4} u^4 \right) \, r^2 \, dr\\
  \quad -  \int_0^{\infty}  \left(\frac{1}{2} u_t^2 + \frac{1}{2}u_r^2  \pm \frac{1}{4} u^4 \right) \, r \chi'_R \, r^2 \, dr - \int_0^\infty u u_r \,  r \chi'_R \, r \, dr
   \end{multline}
where here the bracket $\ang{f \mid g}$ is the radial $L^2(\R^3)$ inner product, $$\ang{f \mid g}:= \int_0^{\infty}  f(r)g(r) \, r^2 \, dr.$$
\end{lem}
\begin{proof}
The proof follows from the equation~\eqref{u eq}  and integration by parts. 
\end{proof}
\begin{rem} \label{r:vir}
In general for a semilinear equation of the form 
\ant{
u_{tt} - \Delta u = \pm  \abs{u}^{p-1} u
}
one has the following formal virial identity 
\begin{equation*}
\frac{d}{dt} \ang{ u_t \mid u + x  \cdot \na u}= - E(u) \pm ( \frac{p-3}{p + 1}) \| u \|_{L^{p + 1}}^{p + 1}.
\end{equation*}
Note that the right-hand-side can be bounded from above by a negative constant times the \emph{conserved} energy in the case $1+ \sqrt{2} <p \le 3$, yielding a monotone quantity. But in the case $3<p<5$ the right-hand-side cannot be controlled by the conserved energy in the case of the focusing equation. However, for the range $3 <p<5$ a different rigidity argument is available, based on the ``channels of energy method" developed by Duyckaerts, Kenig, and Merle,  \cite{DKM4, DKM5}. For an implementation of this strategy for the range $p \in (3, 5)$, see Shen, \cite{Shen}. 

We also note that the virial identities and the argument in this section also readily extend to the non-radial setting. 
\end{rem} 
The proof of Proposition~\ref{rig prop} will now following by applying the above lemma to our pre-compact trajectory $\vec u(t)$ in order to show that the energy must by non-positive. One concludes the proof by noting that a solution to the defocusing equation with non-positive energy must be identically zero. In the case of the focusing equation we recall Proposition~\ref{neg blow up} which says that a solution  with non-positive energy must either be identically zero or blow up in both time directions, and the latter is impossible under the hypothesis of Proposition~\ref{rig prop}. 
\begin{proof}[Proof of Proposition~\ref{rig prop}]
Fix $\eta>0$. We will show that for $\vec u(t)$ as in Proposition~\ref{rig prop} we have 
\EQ{
\E( \vec u) \le C\eta
}
for a fixed content $C$ which is independent of $\eta$. First, note that since $\{\vec u(t) \mid t \in \R\}$ is pre-compact in $( \dot{H}^{\frac{1}{2}} \times \dot{H}^{-\frac{1}{2}} )\cap (\dot{H}^1 \times L^2)( \R^3)$ we can find $R_0= R_0( \eta)$ so that for all $R \ge R_0$ and for all $t \in \R$ we have 
 \EQ{\label{H1 R}
 \int_R^{\infty} (u_t^2(t) + u_r^2(t)) \, r^2 \, dr\le  \eta.
}
Moreover, due to the embeddings $\dot{H}^{\frac{1}{2}} \cap \dot{H}^1 \hookrightarrow \dot{H}^{\frac{3}{4}} \hookrightarrow L^4$ we can choose $R_0(\eta)$ large enough so that we also have 
\EQ{
\int_R^{\infty} u^4(t) \, r^2 \, dr \le \eta 
}
for all $R \ge R_0$ and for all $t \in \R$. Finally, we note that for any $R>0$ and for any smooth radial function in $\dot{H}^1(\R^3)$ we have
\ant{
\int_R^{\infty} f^2(r) \, dr + R  f^2(R)  = - \int_R^{\infty} f_r(r) f(r) \, r \, dr,
}
which can be obtained by integrating by parts. This implies that 
\ant{
\int_R^{\infty} f^2(r) \, dr \le \int_R^{\infty} f_r^2(r) \, r^2 \, dr.
}
Therefore, for our pre-compact trajectory $\vec u(t)$ we can use~\eqref{H1 R} to obtain
\EQ{
\int_R^{\infty} u^2(t, r) \, %r^2 \, 
dr  \le  \eta
}
for all $R \ge R_0(\eta)$ and for all $t \in \R$. Letting $R \ge R_0(\eta)$ we can apply these estimates to the last three terms on the right-hand side of~\eqref{virial}: 
\ant{
 &\abs{\int_0^{\infty}(1- \chi_R)\left(\frac{1}{2} u_t^2 + \frac{1}{2}u_r^2  \pm \frac{1}{4} u^4 \right) \, r^2 \, dr} \le C \eta \\
 &\abs{ \int_0^{\infty}  \left(\frac{1}{2} u_t^2 + \frac{1}{2}u_r^2  \pm \frac{1}{4} u^4 \right) \, r \chi'_R \, r^2 \, dr} \le C\int_R^{2R} \left(\frac{1}{2} u_t^2 + \frac{1}{2}u_r^2  + \frac{1}{4} u^4 \right) \, r^2 \, dr \le C \eta\\
 & \abs{\int_0^\infty u u_r \,  r \chi'_R \, r \, dr} \le \left( \int_R^{2R} u^2_r \, r^2 \, dr \right)^{\frac{1}{2}} \left( \int_R^{2R} u^2  \, dr \right)^{\frac{1}{2}}  \le C \eta.
 }
Inserting the above estimates into~\eqref{virial} and averaging in time from $0$ to $T$  and we obtain the estimate 
\EQ{
E(\vec u) &\le  C \eta +  \frac{1}{T} \abs{ \ang{ u_t(T) \mid \chi_R( u(T) + ru_r(T))}}\\
& \quad + \frac{1}{T}\abs{ \ang{ u_t(0) \mid \chi_R( u(0) + ru_r(0))}}.
}
Now, set $R=T$ above with $T$ large enough so that $T \gg R_0(\eta)$. We have 
\EQ{\label{E T}
E(\vec u) &\le C \eta +  C\frac{1}{T} \int_0^T \abs{u_t(T)} \abs{u(T)} \, r^2 \, dr + C \frac{1}{T} \int_0^T \abs{u_t(0)} \abs{u(0)} \, r^2 \, dr\\
& \quad + C\frac{1}{T} \int_0^T \abs{u_t(T)} \abs{u_r(T)} \, r^3 \, dr + C \frac{1}{T} \int_0^T \abs{u_t(0)} \abs{u_r(0)} \, r^3 \, dr.
}
We estimate the second and third terms on the right-hand-side of~\eqref{E T} by 
\ant{
 \frac{1}{T} \int_0^T \abs{u_t} \abs{u} \, r^2 \, dr  &\le \frac{1}{T}\left( \int_0^T u_t^2 \, r^2 \, dr\right)^{\frac{1}{2}} \left( \int_0^T \abs{u}^3 \, r^2 \, dr\right)^{\frac{1}{3}}  \left( \int_0^T  \, r^2 \, dr\right)^{\frac{1}{6}}\\
 &\le C \frac{1}{T^{\frac{1}{2}}} \| u_t\|_{L^2} \|u\|_{\dot{H}^{\frac{1}{2}}} \to 0 \mas T \to \infty
 }
where in the last line we have used that embedding $\dot{H}^{\frac{1}{2}} \hookrightarrow L^3$,  the fact that the critical element $\vec u(t)$ satisfies  $ \sup_{t \in \R} \|u(t)\|_{\dot H^{\frac{1}{2}}} \lesssim 1$, and   $\sup_{t \in \R} \| u_t \|_{L^2} \lesssim 1$. To estimate the fourth and fifth terms in~\eqref{E T} we note that for $T \gg R(\eta)$ we have 
\ant{
\frac{1}{T} \int_0^T \abs{u_t} \abs{u_r} \, r^3 \, dr  &\le \frac{1}{T} \int_0^{R(\eta)} \abs{u_t} \abs{u_r} \, r^3 \, dr  + \frac{1}{T} \int_{R(\eta)}^T \abs{u_t} \abs{u_r} \, r^3 \, dr \\
&\le  \frac{R(\eta)}{T} \|u_t\|_{L^2}\| u\|_{\dot{H}^1} + \left( \int_{R(\eta)}^T u_t^2 \, r^2 \, dr\right)^{\frac{1}{2}}\left( \int_{R(\eta)}^T u_r^2 \, r^2 \, dr\right)^{\frac{1}{2}}\\
&=  \eta +O(T^{-1}) \mas T \to \infty
}
Thus, letting  $T \to \infty$ in~\eqref{E T} we obtain 
\ant{
E( \vec u) \le C \eta,
}
as desired. Since this holds for all $\eta >0$ we can conclude that 
\EQ{ \label{E=0}
E(\vec u) \le 0.
}
In the case that $\vec u(t)$ is a solution to the de-focusing equation, we are done as we can conclude from~\eqref{E=0} that  $\vec u(t) \equiv 0$. In the case that $\vec u(t)$ is a solution to the focusing equation, we note that~\eqref{E=0} together with Proposition~\ref{neg blow up} imply that either $\vec u(t) \equiv 0$ or $\vec u(t)$ blows up in finite time in both time directions. However, the latter case is impossible as we have assumed that $\vec u(t)$ is global in time. This completes the proof of Proposition~\ref{rig prop}. 
\end{proof}

\section{Proof of Theorem~\ref{main}}

We provide a brief summary of the proof of Theorem~\ref{main}, which is now complete. We argue by contradiction. If Theorem~\ref{main} were false, we could, by Proposition~\ref{crit element},  find a critical element, i.e., a \emph{nonzero} solution $\vec u(t)$ to~\eqref{u eq} with the compactness property in $\dot{H}^{\frac{1}{2}} \times \dot H^{-\frac{1}{2}}$ on an open interval $I \ni0$ with scale $N(t)$. By the remarks following the statement of Proposition~\ref{crit element}, we can reduce to the case $I= (T_-, \infty)$ and $N(t) \le 1$ for $t \in [0, \infty)$. By Theorem~\ref{4.1} we then have 
\ant{
\| \vec u(t) \|_{\dot{H}^1 \times L^2} \lesssim N(t)^{\frac{1}{2}} \mfor t \in [0, \infty)
}
Then, since we are assuming $\vec u(t)$ is nonzero, by Section~\ref{cascade} we can conclude that $N(t) \equiv1$ for all $t \in [0, \infty)$. We can then ensure that $\vec u(t)$ is global-in-time with for all $t \in \R$ and the by Proposition~\ref{h2 bound} we know that $\vec u(t)$ has a pre-compact trajectory in $\dot{H}^{\frac{1}{2}} \times \dot H^{-\frac{1}{2} }\cap \dot{H}^1 \times L^2$. But then Proposition~\ref{rig prop} shows that $\vec u(t) \equiv 0$, which is a contradiction. 

\bibliographystyle{plain}
\bibliography{researchbib}

 \medskip

\centerline{\scshape Benjamin Dodson, Andrew Lawrie}
\smallskip
{\footnotesize
% please put the address of the first author
 \centerline{Department of Mathematics, The University of California, Berkeley}
\centerline{970 Evans Hall \#3840, Berkeley, CA 94720, U.S.A.}
\centerline{\email{ benjadod@berkeley.edu, alawrie@math.berkeley.edu}}
} % Do not forget to end the {\footnotesize by the sign }

\end{document}